\documentclass[a4paper,oneside,11pt]{article}
\usepackage{amsmath}
\usepackage{amssymb}
\usepackage{titlesec}
\usepackage{graphicx}
\usepackage{fancyhdr}
\usepackage{latexsym}
\usepackage{mathrsfs}
\usepackage{booktabs}
\usepackage{mathrsfs}
\usepackage{mathtools}
\usepackage{amsthm}
\usepackage{wasysym}
\usepackage{cite}
\usepackage{color}
\usepackage{geometry}
\numberwithin{equation}{section}

\newtheorem{theorem}{Theorem}[section]
\newtheorem{lemma}{Lemma}[section]

\title{On the asymptotic behavior of the one-dimensional motion of the polytropic ideal gas with degenerate heat conductivity}
\date{  }
\author{}
\date{ }
\author{Guocai Cai$^a$, Yan-Fang Peng$^b$, Yi Peng$^c$\thanks{
  Email addresses: gotry@xmu.edu.cn (G.C. Cai),  pyfang2005@sina.com (Y.F. Peng),  pengyi16@mails.ucas.ac\quad.cn(Y. Peng)}  \\[3mm]  a. School of Mathematical Sciences, Xiamen University,\\
Xiamen 361005, P. R. China; \\
b. School of Mathematical Sciences,   Guizhou Normal University,  \\Guiyang, 550001, P. R. China;  \\c. School of Mathematical Sciences,\\
  University of Chinese Academy of Sciences, Beijing 100049, P. R. China}
\begin{document}
\maketitle

\begin{abstract}

We consider the one dimensional compressible Navier-Stokes system   with  constant viscosity   and the nonlinear  heat conductivity being  proportional to a positive power of the temperature which   may be degenerate.  This problem is imposed on the stress-free boundary condition, which reveals the motion of a viscous heat-conducting perfect polytropic gas   with adiabatic ends   putting into a vacuum.    We prove that  the solution of  one dimensional compressible Navier-Stokes system   with  the stress-free boundary condition  shares the same large-time behavior as the case of constant  heat conductivity.
 \end{abstract}
\noindent{\bf Keywords:}  Compressible Navier-Stokes system; One-dimension;
Degenerate heat conductivity;  Stress-free

\section{Introduction}

In this paper, we investigate the one-dimensional compressible Navier-Stokes system, which reveals the motion of a viscous heat-conducting perfect polytropic gas   and has the following form in the Lagrangian mass coordinates:
\begin{gather}
u_{t}=v_{x}  \label{as1},\\
v_{t}=\left(-R\frac{\theta}{u}+\mu\frac{v_{x}}{u}\right)_{x},\label{as2}\\
c_{\nu}\theta_{t}=-R\frac{\theta v_{x}}{u}+\mu\frac{v_{x}^2}{u}+\left(\frac{\kappa\theta_{x}}{u}\right)_{x},\label{as3}
\end{gather}
where $t>0$ is time, $x\in \Omega=(0,1) $ denotes the Lagrange mass coordinate, the unknown functions $u$, $v$, $\theta$ are the specific volume, velocity, absolute temperature of the gas respectively. The system is studied subject to the given initial data
\begin{gather}
(u,v,\theta)(x,0)=(u_{0},v_{0},\theta_{0})(x),\quad u_{0}>0,\quad \theta_{0}>0,\label{1523}
\end{gather}
and the boundary conditions
\begin{gather}
	\left(-R\frac{\theta}{u}+\mu\frac{v_{x}}{u} \right)(0,t)=\left(-R\frac{\theta}{u}+\mu\frac{v_{x}}{u}\right) (1,t)=0,\quad
	\theta_{x}(0,t)=\theta_{x}(1,t)=0,\label{2153}
\end{gather}
 and certainly the initial data \eqref{1523} should be compatible with the boundary conditions \eqref{2153}. In \eqref{2153}, both specific gas constant $R$ and heat capacity at constant volume $c_{\nu}$ are positive constants, and $\mu$, $\kappa$ satisfy
\begin{gather}
\mu=\tilde{\mu}(1+\alpha u^{-\alpha}), \quad \kappa=\tilde{\kappa}\theta^\beta,\label{uk1}
\end{gather}
with constants $\tilde {\mu}$, $\tilde {\kappa}>0$ and $\alpha$, $\beta \geq 0$.
Without loss of generality, we assume that $\tilde{\mu}=\tilde{\kappa}=R=c_{\nu}=1$.

This system is a model of the one-dimensional motion of the polytropic ideal gas with adiabatic ends  putting into a vacuum, and the condition \eqref{2153} is called the stress-free one. It is easy to find that, when $\alpha=0$, the system has a trivial solution
\begin{gather}
u(x,t)=c(1+t),\quad v(x,t)=c(x-\frac{1}{2}),\quad \theta(x,t)=c,\label{as4}
\end{gather}
corresponding to initial data
\begin{gather}
u_{0}(x)=c,\quad v_{0}(x)=c(x-\frac{1}{2}),\quad \theta_{0}(x)=c,
\end{gather}
where $c$ is a positive constant.

For constant coefficients $(\alpha=\beta=0)$ with large initial data, Kazhikhov and Shelukhin \cite{9} first obtained the global existence of solutions under the following boundary conditions:
 \begin{gather}
v(0,t)=v(1,t)=0,\quad \theta_{x}(0,t)=\theta_{x}(1,t)=0,\label{2523}
\end{gather}
which mean that the gas is confined into a fixed tube with impermeable gas. From then on, significant progress has been made on the mathematical aspect of the initial boundary value problems, see \cite{1,2,3} , \cite{15,18,17} and the references therein.

Kawohl \cite{ssyy11}, Jiang \cite{ssyy22,ssyy33} and Wang \cite{ssyy44} established the global existence of smooth solutions for \eqref{as1}-\eqref{as3}, \eqref{1523} with boundary condition of either \eqref{2153} or \eqref{2523} under the assumption $\mu (u)\geq \mu_{0}>0$
for any $u>0$ and $\kappa$ may depend on both density and temperature. Under the assumption that $\alpha=0$ and $\beta\in (0,3/2)$, Jenssen-Karper
\cite{24} proved the global existence of a weak solution to \eqref{as1}-\eqref{2153}. Later, the global existence of  strong solutions is obtained by  Pan-Zhang \cite{28} when $\alpha=0$ and $\beta\in(0,\infty)$. In \cite{28}, they only consider the case of stress-free and heat insulated boundary conditions. For the case of stress-free and
heat insulated boundary condition, Duan-Guo-Zhu \cite{7} obtain the global strong solutions of \eqref{as1}-\eqref{uk1} under the initial condition that
\begin{gather}
(u_{0},v_{0},\theta_0)\in H^{1}\times H^{2}\times H^{2}.\label{dgz}
\end{gather}

Recently, Huang-Shi-Sun \cite{1212} have proved the existence and uniqueness of global strong solutions to \eqref{as1}-\eqref{uk1} for $\alpha\geq 0$ and $\beta\in(0,\infty)$
with the condition on the initial data
\begin{gather}
(u_{0},v_{0},\theta_{0})\in H^{1}.
\end{gather}
More precisely, they gave the following results:
\newtheorem{myth}{Proposition}[section]
\begin{myth}\label{Proposition1}
(\cite{1212}) Suppose that
\begin{gather}
\alpha \geq 0,\quad  \beta>0, \notag
\end{gather}
and the initial data $(u_{0}, v_{0}, \theta_{0})$ satisfies
\begin{gather}
u_{0}, v_{0}, \theta_{0} \in H^1(0,1),
\end{gather}
and
\begin{gather}
\inf_{x\in (0,1)}u_0(x)>0, \quad \inf_{x\in (0,1)}\theta_0(x)>0.\label{u0v0}
\end{gather}
Then the initial-boundary-value problem \eqref{as1}-\eqref{uk1} has a unique global strong  solution $(u,v,\theta)$
such that for each $T>0$,
\begin{equation}
\left\{ \begin{array}{l}
u,v,\theta \in L^{\infty}(0,T;H^{1}(0,1)),\\ \label{klmk}
u_{t}\in L^{\infty}(0,T;L^{2}(0,1))\cap L^{2}(0,T;H^{1}(0,1)),\\
v_{x},\theta_{x},v_{t},\theta_{t},u_{xt},v_{xx},\theta_{xx}\in L^{2}((0,T)\times(0,1)).
       \end{array} \right.\\
\end{equation}
Moreover,
\begin{gather}
C_{0}^{-1}\leq u(x,t)\leq C_{0}, \quad C_{0}^{-1}\leq \theta(x,t)\leq C_{0},
\end{gather}
where $C_{0}$ is a positive constant depending on the initial data and $T$.
\end{myth}
For the asymptotic behavior with stress-free condition, in \cite{kawa} and \cite{oka}, it indicates that any classical solution which satisfies some restricted assumptions on the initial data and the ratio converges to the state like \eqref{as4}. On the other hand, in the case that $\alpha=\beta=0$, some other asymptotic properties of the solution such as the growth rate of $u$ and $\int_{0}^{1}udx$ is given in \cite{18}. For the system with $\alpha=\beta=0$, Nagasawa \cite{15} showed the convergence of the strong solution and its rate. However, his methods depend crucially on the condition that  the heat-conductivity is  a  constant.

The goal of this paper is to further study the large time behavior of the global strong solutions obtained in \cite{1212}. We focus on the case that the heat conductivity is  proportional to a positive power of the temperature, which is nonlinear and may be degenerate, that is,  $\alpha=0,  \beta\in(0,\infty)$. And we have the following conclusion.
\newtheorem{mythm}{Theorem}[section]
\begin{mythm}\label{thm:light}
Under the conditions of Proposition \ref{Proposition1} with $\alpha=0$ and $\beta>0$, let $(u,v,\theta)$ be (unique) strong solution to \eqref{as1}-\eqref{uk1} satisfying
 \eqref{klmk} for any $T>0$. Then there exist positive constants $C$, $\lambda$ depending only on $\beta$ and initial data such that
\begin{gather}
C^{-1}\leqslant \frac{u(x,t)}{1+t} \leqslant C,\quad C^{-1}\leqslant \theta (x,t) \leqslant C,
\end{gather}
and
\begin{gather}
 \left \lVert\frac{u(x,t)}{1+t}-A,v(x,t)-\int_{0}^{1}v_{0}(x)dx-A(x-\frac{1}{2}),\theta(x,t)-A \right \lVert_{H^1(0,1)}\leq C(1+t)^{-\lambda},
\end{gather} where  $A$ satisfies
\begin{gather}
A\triangleq2\sqrt {36+3\left(2E_{0}-\left(\int_{0}^{1}v_{0}(x)dx\right)^2\right)}-12,\label{A0}
\end{gather}
with
\begin{gather}
E_{0}\triangleq\int_{0}^{1}\left(\frac{1}{2}v_{0}^2(x)+\theta_{0}(x) \right)dx. \label{E0}
\end{gather}
\end{mythm}

In order to prove Theorem \ref{thm:light}, it is convenient to transform the problem into somewhat similar to the outer pressure problem which was discussed in \cite{15}.
First we introduce a variable $ \hat{t}=\log(1+t)$ and a function $\tilde{u}(x,t)=u/(1+t) $. For a function $f(x,t)$, define $\hat{f}$ by
\begin{gather}
\hat{f}=\hat{f}(x,\hat{t})\triangleq f(x,t(\hat{t}))=f(x,e^{\hat{t}}-1). \notag
\end{gather}
 It follows from \eqref{as1}-\eqref{as3} that $(\hat{\tilde{u}},\hat{v},\hat{\theta},\hat{t})$ satisfies
\begin{gather}
\hat{\tilde{u}}_{\hat{t}}+\hat{\tilde{u}}=\hat{v}_{x},\label{1412}\\
\hat{v}_{\hat{t}}=\left(-\frac{\hat{\theta}}{\hat{\tilde{u}}}+ \frac{\hat{v_{x}}}{\hat{\tilde{u}}} \right)_{x},\label{1413}\\
\hat{\theta}_{\hat{t}}=-\frac{\hat{\theta} \hat{v}_{x}}{\hat{\tilde{u}}}+ \frac{\hat{v}^2_{x}}{\hat{\tilde{u}}}+
\left(\frac{\hat{\theta}^\beta \hat{\theta}_{x}}{\hat{\tilde{u}}} \right)_{x}.\label{1414}
\end{gather}
To avoid complicated notation, in what follows, we still write $(\hat{\tilde{u}},\hat{v},\hat{\theta},\hat{t})$
as $(u,v,\theta,t)$, i.e, in the later discussion, the system \eqref{1412}-\eqref{1414} will be written as
\begin{gather}
u_t+u=v_x,\label{412}\\
v_t=\left(-\frac{\theta}{u}+ \frac{v_x}{u} \right)_x,\label{413}\\
\theta_t=-\frac{\theta v_x}{u}+\frac{v^2_x}{u}+
\left(\frac{\theta^\beta \theta_x}{u} \right)_{x}.\label{414}
\end{gather}
However, for the sake of the distinction of these symbols, we will point out the equations they satisfy.

It is easy to see that the system \eqref{412}-\eqref{414} with the same initial and boundary conditions \eqref{1523} and \eqref{2153} has the energy identity
\begin{gather}
\int_{0}^{1}\left(\frac{1}{2}v^{2}(x,t)+\theta(x,t) \right)dx=E_{0},\label{66}
\end{gather}
where $E_0$ is given by \eqref{E0}.

For further transformation, we set
\begin{gather}
w=v(x,t)-\int_{0}^{1}v_{0}(\xi)d\xi -\int_{0}^{x}u(\xi,t)d\xi +\int_{0}^{1} \int_{0}^{\xi}u(\eta,t)d\eta d\xi,\label{004}
\end{gather}
where $u$ and $v$ are shown in the system \eqref{412}-\eqref{414}.
One can check that
\begin{gather}
w_x=v_x-u, \label{wvxu}
\end{gather}
 and
\begin{gather}
v_t=w_t+w \label{wvtu},
\end{gather}
where in the last equality we have utilize \eqref{412} and the fact
\begin{gather}\label{w0}
\int_{0}^{1}w(x,t)dx =0,
\end{gather}
due to \eqref{413} and \eqref{2153}.
Therefore, if we use $w(x,t)$ instead of $v(x,t)$,  the system \eqref{412}-\eqref{414} becomes
\begin{gather}
u_{t}=w_{x},\label{001} \\
w_{t}+w=\left(-\frac{\theta}{u}+ \frac{w_{x}}{u} \right)_{x},\label{002} \\
\theta_{t}= \frac{( w_{x}+ u- \theta)^2}{u}+\frac{( w_{x}+ u- \theta)\theta}{u}+\left(\frac{\theta^{\beta} \theta_{x}}{u} \right)_{x}.\label{003}
\end{gather}
And by initial and boundary conditions \eqref{1523}-\eqref{2153}, we have
\begin{gather}
(u,w,\theta)(x,0)=(u_{0},w_{0},\theta_{0}),\quad u_{0}>0, \quad \theta_{0}>0,\quad \int_{0}^{1}w_{0}dx=0,\\
\left(-\frac{\theta}{u}+\frac{w_{x}}{u} \right)(0,t)=\left(-\frac{\theta}{u}+ \frac{w_{x}}{u} \right)(1,t)=-1,\label{1212}\\
\theta_{x}(0,t)=\theta_{x}(1,t)=0.\label{11212}
\end{gather}
Therefore, we reduce the problem \eqref{as1}-\eqref{2153} to the outer pressure problem \eqref{001}-\eqref{11212}, and this transformation is crucial in our following analysis.

We now explain the main strategy of this paper. To establish Theorem 1.1, the key point is to control the lower and upper bounds of both
$u$ and $\theta$. To this end, we mainly employ some ideas in \cite{15,LIJING,2202}. However, due to the appearance of the degeneracy and nonlinearity of the heat conductivity duo to $\beta>0$ and the boundary conditions of \eqref{1212}, some deeper discussions are much more involved here. Our discussion bases on three  treatments. First, motivated by \cite{9}, we obtain an explicit expression of $u$ which is used to estimate the lower and upper bounds of $u$. Next, multiplying the equation of the temperature by $(\theta^{-\frac{1}{2}}-\alpha_{1}^{-\frac{1}{2}})^{p}_{+}\theta^{-\frac{3}{2}}$ and using the Gronwall's inequality, we succeed in obtaining that the temperature is indeed bounded from below which plays an important role in our further analysis. Moreover, we prove the following significant estimate :
\begin{gather*}
\int_{0}^{T}\int_{0}^{1}\theta^{-1}\theta_{x}^{2}dxdt\leq C,
\end{gather*}
which not only implies that
\begin{gather*}
\int_{0}^{T}\int_{0}^{1}\theta^{-1}\theta_{x}^{2}dxdt+\int_{0}^{T}\int_{0}^{1}(w_{x}+u-\theta)^{2}dxdt\leq C,
\end{gather*}
but also gives that  $L^2((0,T)\times (0,1))$-norm of $\theta_x$ can be  bounded  by the
 $L^4(0,T;L^2(0,1))$-norm of $w_{x}+u-\theta$ which is crucial for achieving   the uniform bound on $L^2((0,T)\times (0,1))$-norm
 of $(w_{x}+u-\theta)_{x}$ (see Lemma 2.6). Finally, by modifying some ideas in \cite{LIJING}, \cite{15} and \cite{2202}, we  obtain the  exponential decay of strong solutions. We should  point   that in the proof of Theorem 1.1, cut-off techniques play an important role  in many places which are greatly different from \cite{2202} and \cite{15}.  The whole procedure will be carried out in the next section.
\newpage

\section{Proof of Theorem 1.1}
In this section, we will give a priori estimates for  \eqref{001}-\eqref{11212} in $[0,1]\times [0,T]$ which is derived by the system \eqref{412}-\eqref{414} with initial and boundary conditions \eqref{1523} and \eqref{2153} through \eqref{004}. Unless otherwise stated, we follow the convention that $C$ is an unspecified positive constant that may vary from expression to expression, even across an inequality (but
not across an equality). Also $C$ depends on initial data and $\beta$ generally but does not depend on $T$, and
the dependence of $C$ on other parameters will be specified within parenthesis when necessary.

We begin with the following energetic estimates.
\begin{lemma}\label{2.1}
For the system \eqref{412}-\eqref{414} in $[0,1]\times [0,T]$ with \eqref{1523} and \eqref{2153}, we have
\begin{gather}
\int_{0}^{1}\left(\frac{v^{2}}{2}+\bigl(\theta-\log\theta-1\bigr) \right)dx+\int_{0}^{T}V(s)ds \leq e_{0},\label{88}
\end{gather}
where $e_{0}$ is a positive constant depending only on initial data, and
\begin{gather}
V(t)\triangleq\int_{0}^{1}\left(\frac{\theta^{\beta} \theta_{x}^2}{u\theta^2}+\frac{(w_{x}+u-\theta)^2}{u\theta} \right)dx.\label{Vt}
\end{gather}
\end{lemma}
\begin{proof}
First, by \eqref{413} and \eqref{2153}, we have
\begin{gather}
\int_{0}^{1} \frac{v_x-\theta}{u} dx=\frac{d}{dt}\int_{0}^{1}\int_{0}^{x}v(\xi,t)d\xi dx.\label{v23}
\end{gather}
Next, by Cauchy's inequality, we get
\begin{gather}
\left|\int_{0}^{x}vd\xi \right|\leq \int_{0}^{1}|v|d\xi\leq \left(\int_{0}^{1}v^{2}d\xi \label{v24} \right)^{\frac{1}{2}}.
\end{gather}
Finally, multiplying \eqref{414} by $1/\theta $, integrating with respect to $x$ over $[0,1]$, by \eqref{wvxu} and \eqref{v23}, we have
\begin{gather}
-\frac{d}{dt}\int_{0}^{1} \log \theta dx+ V(t)=-\frac{d}{dt}\int_{0}^{1}\int_{0}^{x}v(\xi,t)d\xi dx,\label{23}
\end{gather}
which along with \eqref{66} and \eqref{v24} yields \eqref{88}.
\end{proof}
The following Lemma gives the  lower  and upper bounds of $u$, which is a key for our further analysis.
\begin{lemma}
There exists a positive constant $M$ depending only on initial data and $\beta$ such that for any $(x,t)\in [0,1]\times [0,\infty)$,
\begin{gather}
M^{-1}\leq u(x,t)\leq M.\label{008}
\end{gather}
\end{lemma}
\begin{proof}
Integrating \eqref{413} over $[0,x]$, by \eqref{412} and \eqref{2153}, we get
\begin{gather}
u_t+\left(-\int_{0}^{x}(v-v_{0})d\xi+t\right)_tu=\theta.
\end{gather}
Consequently,
\begin{gather}
u(x,t)=\frac{u_{0}(x)B(x,t)}{e^{t}}\left(1+ \frac{1}{u_{0}}\int_{0}^{t}\frac{\theta e^{s}}{B(x,s)}ds \right).\label{8}
\end{gather}
where
\begin{gather}
B(x,t)\triangleq e^{\int_{0}^{x} (v(\xi,t)-v_{0}(\xi )) d\xi}.\label{Bxt}
\end{gather}

By \eqref{88} and \eqref{v24},
\begin{gather}
C^{-1}\leq B(x,t)\leq C.\label{77}
\end{gather}
Now we set
\begin{gather}
\bar{\theta}\triangleq \int_{0}^{1}\theta(x,t)dx,
\end{gather}
by Jensen's inequality and \eqref{88}, we see that
\begin{gather}
0<\alpha_{1}\leq \bar{\theta}\leq \alpha_{2}<\infty,\label{123}
\end{gather}
where $\alpha_{1}$ and $\alpha_{2}$ are two roots of
\begin{gather}
x-\log x-1=e_{0}.\notag
\end{gather}
By Cauchy's inequality and the definition of $V(t)$ (see \eqref{Vt}),
\begin{equation}
\begin{split}
\left|\theta^{\frac{\beta+1}{2}}-\bar{\theta}^{\frac{\beta+1}{2}} \right| &\leq \frac{\beta+1}{2}\int_{0}^{1}\theta^\frac{\beta-1}{2}|\theta_x|dx\\
&\leq
 \frac{\beta+1}{2}\left(\int_{0}^{1}\frac{\theta^{\beta}\theta_{x}^{2}}{u\theta^{2}}dx \right)^{\frac{1}{2}}\left(\int_{0}^{1}u\theta dx\right)^{\frac{1}{2}}\\
 &\leq CV^{\frac{1}{2}}(t)\max_{x\in [0,1]}u^{\frac{1}{2}}(x,t),
\end{split}
\end{equation}
which together with \eqref{123}, implies that,  for all $(x,t)\in [0,1]\times [0,\infty)$,
\begin{gather}
\frac{\alpha_{1}}{4}-CV(t)\max_{x\in [0,1]}u(x,t)\leq \theta(x,t)\leq C+CV(t)\max_{x\in [0,1]}u(x,t).\label{9}
\end{gather}
Therefore, it follows from \eqref{8}, \eqref{77} and \eqref{9} that
\begin{equation}
\begin{split}
u(x,t)&\leq C+Ce^{-t}\int_{0}^{t}e^{s}\max_{x\in [0,1]}\theta(x,s)ds\\
&\leq C+C\int_{0}^{t}V(s)\max_{x\in [0,1]}u(x,s)ds.
\end{split}
\end{equation}
By \eqref{88} and Gronwall's inequality, for all $(x,t)\in [0,1]\times [0,\infty)$,
\begin{gather}
u(x,t)\leq C.\label{63}
\end{gather}

On the other hand, we deduce from \eqref{8}, \eqref{9} and \eqref{63} that for any $x\in [0,1]$,
\begin{equation}
\begin{split}
u(x,t)&\geq Ce^{-t}\int_{0}^{t}e^{s}\theta(x,s)ds\\
&\geq \frac{C\alpha_{1}}{4}(1-e^{-t})-C\int_{0}^{t}e^{s-t}V(s)ds.
\end{split}
\end{equation}

Noticing that, by \eqref{88},
\begin{equation}
\int_{0}^{t}e^{s-t}V(s)ds=\int_{0}^{\frac{t}{2}}e^{s-t}V(s)ds+\int_{\frac{t}{2}}^{t}e^{s-t}V(s)ds\rightarrow 0,
\end{equation}
as $t\rightarrow\infty.$ As a result, there exists a $T^{\ast}>0$ such that for $t\in [T^{\ast},\infty)$,
\begin{equation}
\begin{split}
u(x,t)\geq \frac{C\alpha_{1}}{8}.\label{888}
\end{split}
\end{equation}
For $t\in [0,T^{\ast}]$, it follows from \eqref{8} and \eqref{77} that  $u(x,t)\geq Ce^{-T^{\ast}}$, which together with \eqref{63} and \eqref{888} gives \eqref{008}.
\end{proof}

\begin{lemma}\label{thbelow}
There exists a positive constant $\widehat{N}$ depending only on initial data and $\beta$  such that
\begin{gather}
\theta(x,t) \geq \widehat{N}^{-1},\label{563}
\end{gather}
for any $(x,t)\in [0,1]\times [0,\infty)$.
\end{lemma}
\begin{proof}
For $p>2$ , multiplying \eqref{003} by $(\theta^{-\frac{1}{2}}-\alpha_{1}^{-\frac{1}{2}})_{+}^{p}\theta^{-\frac{3}{2}}$
and integrating by parts, we have
\begin{equation}\label{tha1}
\begin{split}
&\frac{2}{p+1}\left(\int_{0}^{1}\Bigl(\theta^{-\frac{1}{2}}-\alpha_{1}^{-\frac{1}{2}}\Bigr)^{p+1}_{+}dx \right)_{t}+
\int_{0}^{1}\frac{\bigl(w_{x}+u-\theta\bigr)^{2}
\bigl(\theta^{-\frac{1}{2}}-\alpha_{1}^{-\frac{1}{2}}\bigr)^p_{+}}{u\theta^{\frac{3}{2}}}dx\\
&\leq -\int_{0}^{1}\frac{\bigl(w_{x}+u-\theta\bigr)\bigl(\theta^{-\frac{1}{2}}-\alpha_{1}^{-\frac{1}{2}}\bigr)^p_{+}}{u\theta^{\frac{1}{2}}}dx\\
&\leq \frac{1}{2}\int_{0}^{1}\frac{\bigl(w_{x}+u-\theta\bigr)^{2}\bigl(\theta^{-\frac{1}{2}}-\alpha_{1}^{-\frac{1}{2}}\bigr)^p_{+}}{u\theta^{\frac{3}{2}}}dx
+C\int_{0}^{1}\theta^{\frac{1}{2}}\bigl(\theta^{-\frac{1}{2}}-\alpha_{1}^{-\frac{1}{2}}\bigr)^p_{+}dx\\
&\leq \frac{1}{2}\int_{0}^{1}\frac{\bigl(w_{x}+u-\theta\bigr)^{2}\bigl(\theta^{-\frac{1}{2}}-\alpha_{1}^{-\frac{1}{2}}\bigr)^p_{+}}{u\theta^{\frac{3}{2}}}dx\\
&\quad +C\max_{x\in [0,1]}\bigr(\alpha_{1}^{\frac{1}{2}}-\theta^{\frac{1}{2}}\bigr)^2_{+}
\left(1+\int_{0}^{1}\bigl(\theta^{-\frac{1}{2}}-\alpha_{1}^{-\frac{1}{2}}\bigr)^{p+1}_{+}dx \right),
\end{split}
\end{equation}
where in the last inequality we have used Young's inequality and the simple fact
\begin{equation}
\begin{split}
&\theta^{\frac{1}{2}}  \Big(\theta^{-\frac{1}{2}}-\alpha_{1}^{-\frac{1}{2}}\Big)_{+}^p\\
&=\alpha_{1}^{-1}\left( \Big(\theta^{-\frac{1}{2}}-\alpha_{1}^{-\frac{1}{2}}\Big)_{+}^{p-1}+\alpha_{1}^{-\frac{1}{2}}   \Big(\theta^{-\frac{1}{2}}-\alpha_{1}^{-\frac{1}{2}}\Big)_{+}^{p-2}\right)\Big(\alpha_1^{\frac{1}{2}}-\theta^{\frac{1}{2}}\Big)_{+}^2.
\end{split}
\end{equation}
On the other hand, by \eqref{008} and \eqref{123}, we obtain that for any $\beta \in (0,1)$ ,
\begin{equation*}
\begin{split}
\max_{x\in [0,1]}\bigl(\alpha_{1}^{\frac{1}{2}}-\theta^{\frac{1}{2}}\bigr)_{+}&\leq C\int_{0}^{1}\theta^{-\frac{1}{2}}|\theta_{x}|dx\\
&\leq C\left(\int_{0}^{1}\theta^{\beta-2}\theta_{x}^{2}dx \right)^{\frac{1}{2}}\left(\int_{0}^{1}\theta^{1-\beta}dx \right)^{\frac{1}{2}}\\
&\leq CV^{\frac{1}{2}}(t),
\end{split}
\end{equation*}
and that for $\beta \geq 1$ ,
\begin{equation*}
\begin{split}
\max_{x\in [0,1]}\bigl(\alpha_{1}^{\frac{1}{2}}-\theta^{\frac{1}{2}}\bigr)_{+}&\leq C\max_{x\in [0,1]}\bigl(\alpha_{1}^{\frac{\beta}{2}}-\theta^{\frac{\beta}{2}}\bigr)_{+}\\
&\leq C\int_{0}^{1}\theta^{\frac{\beta}{2}-1}|\theta_{x}|dx\\
&\leq C\left(\int_{0}^{1}\frac{\theta^{\beta}\theta_{x}^{2} }{u\theta^{2}}dx \right)^{\frac{1}{2}}\\
&\leq CV^{\frac{1}{2}}(t).
\end{split}
\end{equation*}
In a word, for any $\beta > 0$,
\begin{equation}
\max_{x\in [0,1]}\bigl(\alpha_{1}^{\frac{1}{2}}-\theta^{\frac{1}{2}}\bigr)^2_{+}\leq CV(t).\label{tha2}
\end{equation}
Therefore, we deduce from \eqref{tha1} and \eqref{tha2} that
\begin{equation}
\begin{split}
&\frac{4}{p+1}\left(\int_{0}^{1}\Bigl(\theta^{-\frac{1}{2}}-\alpha_{1}^{-\frac{1}{2}}\Bigr)^{p+1}_{+}dx \right)_{t}+
\int_{0}^{1}\frac{\bigl(w_{x}+u-\theta\bigr)^{2}
\bigl(\theta^{-\frac{1}{2}}-\alpha_{1}^{-\frac{1}{2}}\bigr)^p_{+}}{u\theta^{\frac{3}{2}}}dx\\
&\leq CV(t) \left(1+\int_{0}^{1}\bigl(\theta^{-\frac{1}{2}}-\alpha_{1}^{-\frac{1}{2}}\bigr)^{p+1}_{+}dx \right),
\end{split}
\end{equation}
which, by Gronwall's inequality and \eqref{88}, implies that
\begin{gather}
\parallel\theta^{-\frac{1}{2}}-\alpha_{1} ^{-\frac{1}{2}} \parallel_{L^{p+1}[0,1]} \leq C.\notag
\end{gather}
Letting $p\to +\infty$, we establish \eqref{563} and finish the proof of  Lemma \ref{thbelow}.
\end{proof}

\begin{lemma}
There exists a positive constant C such that
\begin{gather}
\sup_{0\leq t \leq T}\int_{0}^{1}u_{x}^{2}dx+\int_{0}^{T}\int_{0}^{1}\frac{\theta_{x}^{2}}{\theta}dxdt\leq C,\label{541}
\end{gather}
for any $T>0$.
\end{lemma}
\begin{proof}
First, for $\beta\in (0,1)$, multiplying \eqref{003} by $\bigl(\theta^{(1-\beta)/2}-\alpha_{2}^{(1-\beta)/2}\bigr)_{+}\theta^{-(1+\beta)/2}$ gives
\begin{equation}\label{tha22}
\begin{split}
&-\frac{1}{1-\beta}\left(\int_{0}^{1}\bigl(\theta^{(1-\beta)/2}-\alpha_{2}^{(1-\beta)/2}\bigr)_{+}^{2}dx \right)_{t}
+\beta\int_{0}^{1}\frac{\theta_{x}^{2}}{u\theta}1_{(\theta>\alpha_{2})}dx\\
&+\int_{0}^{1}\frac{\bigl(w_{x}+u-\theta\bigr)^{2}\bigl(\theta^{(1-\beta)/2}-\alpha_{2}^{(1-\beta)/2}\bigr)_{+}}{u\theta^{(1+\beta)/2}}dx\\
&=-\int_{0}^{1}\frac{\bigl(w_{x}+u-\theta\bigr)\bigl(\theta^{(1-\beta)/2}-\alpha_{2}^{(1-\beta)/2}\bigr)_{+}}{u\theta^{(\beta-1)/2}}dx
+\frac{1+\beta}{2}\alpha_{2}^{\frac{1-\beta}{2}}\int_{0}^{1}\frac{\theta^{\frac{\beta-3}{2}}\theta_{x}^{2}}{u}1_{(\theta>\alpha_{2})}dx\\
&\leq \frac{\beta}{2}\int_{0}^{1}\frac{\theta_{x}^{2}}{u\theta}1_{(\theta>\alpha_{2})}dx+C\int_{0}^{1}\frac{\theta^{\beta}\theta_{x}^{2}}{u\theta^{2}}dx
-\int_{0}^{1}\frac{\bigl(w_{x}+u-\theta\bigr)\bigl(\theta^{(1-\beta)/2}-\alpha_{2}^{(1-\beta)/2}\bigr)_{+}}{u\theta^{(\beta-1)/2}}dx\\
&\leq \frac{\beta}{2}\int_{0}^{1}\frac{\theta_{x}^{2}}{u\theta}1_{(\theta>\alpha_{2})}dx+ C\varepsilon \int_{0}^{1}\frac{\bigl(w_{x}+u-\theta\bigr)^{2}\bigl(\theta^{(1-\beta)/2}-\alpha_{2}^{(1-\beta)/2}\bigr)_{+}}{u\theta^{(1+\beta)/2}}dx+C(\varepsilon)V(t),
\end{split}
\end{equation}
where we have used \eqref{11212}, \eqref{008}, Young's inequality and the following two estimates
\begin{equation}
\begin{split}
&\int_{0}^{1}\frac{|w_{x}+u-\theta|\bigl(\theta^{(1-\beta)/2}-\alpha_{2}^{(1-\beta)/2}\bigr)_{+}}{u\theta^{(\beta-1)/2}}dx\\
&=\int_{0}^{1}\frac{|w_{x}+u-\theta|\bigl(\theta^{(1-\beta)/2}-\alpha_{2}^{(1-\beta)/2}\bigr)}{u\theta^{(\beta-1)/2}}\left(1_{(\theta>2\alpha_{2})}
+1_{(2\alpha_{2}\geq\theta>\alpha_{2})}\right)dx\\
&\leq \varepsilon\int_{0}^{1}\frac{\bigl(w_{x}+u-\theta\bigr)^{2}}
{u\theta^{\beta}}1_{(\theta>2\alpha_{2})}dx+C(\varepsilon)\max_{x\in[0,1]}
\bigl(\theta^{(1-\beta)/2}-\alpha_{2}^{(1-\beta)/2}\bigr)_{+}^{2}\int_{0}^{1}\theta dx\\
&\quad +C\int_{0}^{1}\frac{\bigl(w_{x}+u-\theta\bigr)^{2}}{u\theta}dx+C\int_{0}^{1}\bigl(\theta^{(1-\beta)/2}-
\alpha_{2}^{(1-\beta)/2}\bigr)_{+}^{2}\theta^{2-\beta}1_{(\theta\leq 2\alpha_{2})}dx\\
&\leq C\varepsilon\int_{0}^{1}\frac{\bigl(w_{x}+u-\theta\bigr)^{2}\bigl(\theta^{(1-\beta)/2}-
\alpha_{2}^{(1-\beta)/2}\bigr)_{+}}{u\theta^{(1+\beta)/2}}dx+C(\varepsilon)V(t),
\end{split}
\end{equation}
and
\begin{equation}
\begin{split}
\max_{x\in [0,1]} \bigl(\theta^{(1-\beta)/2}-\alpha_{2}^{(1-\beta)/2}\bigr)_{+}^{2}&\leq C\max_{x\in [0,1]}
\bigl(\theta^{\frac{1}{2}}-\alpha_{2}^{\frac{1}{2}}\bigr)^{2}_{+}\\
&\leq CV(t).
\end{split}
\end{equation}
Therefore, for $\beta\in (0,1) $ , choosing $\varepsilon$ suitably small, we deduce from \eqref{tha22} that
\begin{equation}
-\frac{2}{1-\beta}\left(\int_{0}^{1}\bigl(\theta^{(1-\beta)/2}-\alpha_{2}^{(1-\beta)/2}\bigr)_{+}^{2}dx \right)_{t}
+\beta\int_{0}^{1}\frac{\theta_{x}^{2}}{u\theta}1_{(\theta>\alpha_{2})}dx
\leq CV(t).
\end{equation}
Integrating this over $[0,T]$, together with  \eqref{008} and \eqref{88}, we check that
\begin{gather}
\int_{0}^{T}\int_{0}^{1}\theta^{-1}\theta_{x}^{2}dxdt\leq C,\label{202}
\end{gather}
which still holds for the case that $\beta \geq 1$ due to \eqref{88} and \eqref{563}.
Furthermore, we have
\begin{gather}
\int_{0}^{T}\max_{x\in [0,1]}(\theta(x,t)-\bar{\theta})^{2}dt\leq \int_{0}^{T}\left(\int_{0}^{1}\theta^{-1}\theta_{x}^{2}dx \right)
\left(\int_{0}^{1}\theta dx \right)dt\leq C.
\end{gather}
We rewrite \eqref{002} as
\begin{gather}
\left(w-\frac{u_{x}}{u} \right)_{t}=-w-\frac{\theta_{x}}{u}+\frac{\theta u_{x}}{u^{2}},
\end{gather}
multiplying  this by $(w-\frac{u_{x}}{u})$ and integrating over $[0,1]$, by \eqref{008}, \eqref{563} and Young's inequality, we get
\begin{equation}
\begin{split}
&\frac{1}{2}\frac{d}{dt}\int_{0}^{1}\left(w-\frac{u_{x}}{u} \right)^{2}dx+\int_{0}^{1}\left(w^{2}+\frac{\theta u_{x}^{2}}{u^{3}} \right)dx\\
&=\int_{0}^{1}\left(\frac{wu_{x}}{u}-\frac{w\theta_{x}}{u}+\frac{\theta_{x}u_{x}}{u^{2}}+\frac{w\theta u_{x}}{u^{2}} \right)dx\\
&\leq \frac{1}{2}\int_{0}^{1}\frac{\theta u_{x}^{2}}{u^{3}}dx+C\int_{0}^{1}\left(\frac{\theta_{x}^{2}}{\theta}+\theta w^{2} \right)dx,\label{62}
\end{split}
\end{equation}
which together with  \eqref{008} and \eqref{563} yields
\begin{equation}
\begin{split}
&\frac{d}{dt}\int_{0}^{1}\left(w-\frac{u_{x}}{u} \right)^{2}dx+\varepsilon_0 \int_{0}^{1}\left(w-\frac{u_{x}}{u} \right)^{2}dx\\
&\leq C\int_{0}^{1}\left(\frac{\theta_{x}^{2}}{\theta}+\theta w^{2} \right)dx,\label{24}
\end{split}
\end{equation}
for $\varepsilon_0\triangleq \min\{1,\frac{1}{2M\widehat{N}}\}$ depending only on initial data and $\beta$. As a result, by \eqref{202} , we obtain
\begin{equation}\label{wux1}
\begin{split}
&\int_{0}^{1}\left(w-\frac{u_{x}}{u} \right)^{2}dx\\
&\leq C\int_{0}^{t}e^{\varepsilon_0(s-t)}\int_{0}^{1}\theta w^{2}dxds+C\int_{0}^{t}e^{\varepsilon_0(s-t)}\int_{0}^{1}\frac{\theta_{x}^{2}}{\theta} dxds\\
&\leq C,
\end{split}
\end{equation}
where we have used the following facts
\begin{gather}
\max_{x\in [0,1]}\theta\leq C\left(1+\int_{0}^{1}\frac{\theta_{x}^{2}}{\theta}dx \right),
\end{gather}
and
\begin{gather}
\int_{0}^{1}w^{2}dx\leq C,\label{w2}
\end{gather}
due to \eqref{123}, \eqref{004}, \eqref{88} and \eqref{008}. Since
$$\int_{0}^{1}\frac{u_{x}^2}{u^2}dx\leq 2\int_{0}^{1}\left(w-\frac{u_{x}}{u} \right)^{2}dx+2\int_{0}^{1}w^{2}dx,$$
together with \eqref{wux1}, \eqref{w2} and \eqref{008}, we have $\int_{0}^{1}u_{x}^{2}dx\leq C$ and finish the proof.
\end{proof}

\begin{lemma}
There exists a positive constant C such that
\begin{gather}
\int_{0}^{T}\int_{0}^{1}\theta^{\beta}\theta_{x}^{2}dxdt\leq C+C\int_{0}^{T}\left(\int_{0}^{1}\bigl(w_{x}+u-\theta\bigr)^{2}dx \right)^{2}dt,\label{7487}
\end{gather}
for any $T>0$.
\end{lemma}
\begin{proof}
First, multiplying \eqref{003} by $(1-\alpha_{2}/\theta)_{+}$, and by \eqref{123}, we obtain
\begin{equation}\label{wxuth1}
\begin{split}
&-\frac{d}{dt}\int_{0}^{1}\Big(\alpha_{2}(\theta -\log \theta-1)+(1-\alpha_{2})\theta\Big)\,1_{(\theta>\alpha_{2})} dx+\int_{0}^{1}\frac{(w_{x}+u-\theta)^{2}(\theta-\alpha_{2})_{+}}{u\theta}dx\\
&\leq C\int_{0}^{1}\frac{\theta^{\beta}\theta_{x}^{2}}{u\theta^{2}}dx+C\int_{0}^{1}\frac{(w_{x}+u-\theta)^{2}}{u\theta}dx+
C\int_{0}^{1}(\theta-\alpha_{2})_{+}^{2}\theta dx\\
&\leq CV(t)+C\max_{x\in [0,1]}(\theta-\alpha_{2})_{+}^{2}\\
&\leq CV(t)+C\int_{0}^{1}\frac{\theta_{x}^{2}}{\theta}dx,
\end{split}
\end{equation}
where in the last inequality we have used the fact
\begin{gather}
\max_{x\in [0,1]}(\theta-\alpha_{2})_{+}^{2}\leq C\int_{0}^{1}\frac{\theta_{x}^{2}}{\theta}dx.
\end{gather}
Integrating \eqref{wxuth1} over $[0,T]$, together with \eqref{88},  \eqref{008} and \eqref{541}, we get
\begin{gather}
\int_{0}^{T}\int_{0}^{1}(w_{x}+u-\theta)^{2}dxdt\leq C.\label{7412}
\end{gather}
Next, multiplying \eqref{003} by $(\theta-\alpha_{2})_{+}$ and using Young's inequality, we have
\begin{equation}
\begin{split}
&\frac{1}{2}\frac{d}{dt}\int_{0}^{1}(\theta-\alpha_{2})_{+}^{2}dx+\int_{0}^{1}\frac{\theta^{\beta}\theta_{x}^{2}}{u}1_{(\theta>\alpha_{2})}dx\\
&=\int_{0}^{1}\frac{(w_{x}+u-\theta)^{2}(\theta-\alpha_{2})_{+}}{u}dx+\int_{0}^{1}\frac{(w_{x}+u-\theta)\theta(\theta-\alpha_{2})_{+}}{u}dx\\
&\leq C\max_{x\in [0,1]}(\theta-\alpha_{2})_{+}\int_{0}^{1}(w_{x}+u-\theta)^{2}dx+C\int_{0}^{1}(\theta-\alpha_{2})_{+}^{2}|w_{x}+u-\theta|dx\\
&\quad +C\int_{0}^{1}(\theta-\alpha_{2})_{+}\,|w_{x}+u-\theta|dx\\
&\leq C\max_{x\in [0,1]}(\theta-\alpha_{2})_{+}^{2}+C\left(\int_{0}^{1}(w_{x}+u-\theta)^{2}dx \right)^{2}+C\int_{0}^{1}(w_{x}+u-\theta)^{2}dx\\
&\quad +C\int_{0}^{1}(\theta-\alpha_{2})_{+}^{2}dx\int_{0}^{1}(w_{x}+u-\theta)^{2}dx\\
&\leq C\int_{0}^{1}\frac{\theta_{x}^{2}}{\theta}dx+C\left(\int_{0}^{1}(w_{x}+u-\theta)^{2}dx \right)^{2}+C\int_{0}^{1}(w_{x}+u-\theta)^{2}dx\\
&\quad +C\int_{0}^{1}(\theta-\alpha_{2})_{+}^{2}dx\int_{0}^{1}(w_{x}+u-\theta)^{2}dx.
\end{split}
\end{equation}
By Gronwall's inequality, we deduce from \eqref{008}, \eqref{541} and \eqref{7412} that
\begin{equation}
\begin{split}
&\sup_{t\in[0,T]}\int_{0}^{1}(\theta-\alpha_{2})_{+}^{2}dx+\int_{0}^{T}\int_{0}^{1}\theta^{\beta}\theta_{x}^{2}1_{(\theta>\alpha_{2})}dxdt\\
&\leq C+C\int_{0}^{T}\left(\int_{0}^{1}(w_{x}+u-\theta)^{2}dx \right)^{2}dt,
\end{split}
\end{equation}
which together with \eqref{88} and \eqref{008} leads to \eqref{7487}.
\end{proof}

\begin{lemma}
There exists a positive constant C such that
\begin{gather}
\int_{0}^{1}(w_{x}+u-\theta)^{2}dx+\int_{0}^{T}\int_{0}^{1}\Bigl((w_{xx}+u_x-\theta_x)^{2}+\theta(w_{x}+u-\theta)^{2}\Bigr)dxdt\leq C , \label{009}
\end{gather}
for any $T>0$.
\end{lemma}
\begin{proof}
Multiplying \eqref{002} by $(w_{x}+u-\theta)_{x}$, and integrating over $[0,1]$, by \eqref{001}-\eqref{11212} except \eqref{002}, we have
\begin{equation}
\begin{split}
&\int_{0}^{1}\frac{(w_{xx}+u_x-\theta_x)^{2}}{u}dx\\
=&-\int_{0}^{1}(w_{x}+u-\theta)_{t}(w_{x}+u-\theta)dx+\int_{0}^{1}(u-\theta)_{t}(w_{x}+u-\theta)dx\\
&-\int_{0}^{1}w_{x}(w_{x}+u-\theta)dx+\int_{0}^{1}\frac{(w_{x}+u-\theta)(w_{x}+u-\theta)_{x}u_{x}}{u^{2}}dx\\
=&-\frac{1}{2}\frac{d}{dt}\int_{0}^{1}(w_{x}+u-\theta)^{2}dx-\int_{0}^{1}\frac{(w_{x}+u-\theta)^{3}}{u}dx-\int_{0}^{1}\frac{\theta(w_{x}+u-\theta)^{2}}{u}dx\\
&+\int_{0}^{1}\frac{\theta^{\beta}\theta_{x}(w_{x}+u-\theta)_{x}}{u}dx+\int_{0}^{1}\frac{(w_{x}+u-\theta)(w_{x}+u-\theta)_{x}u_{x}}{u^{2}}dx.
\end{split}
\end{equation}
Hence,
\begin{equation}\label{wxuth2}
\begin{split}
&\frac{1}{2}\frac{d}{dt}\int_{0}^{1}(w_{x}+u-\theta)^{2}dx+\int_{0}^{1}\left(\frac{\theta(w_{x}+u-\theta)^{2}}{u} +\frac{(w_{xx}+u_x-\theta_x)^{2}}{u}\right)dx\\
=&-\int_{0}^{1}\frac{(w_{x}+u-\theta)^{3}}{u}dx+\int_{0}^{1}\frac{(w_{x}+u-\theta)(w_{x}+u-\theta)_{x}u_{x}}{u^{2}}dx\\
&+\int_{0}^{1}\frac{\theta^{\beta}\theta_{x}(w_{x}+u-\theta)_{x}}{u}dx\\
=& I_{1}+I_{2}+I_{3}.
\end{split}
\end{equation}
A straightforward calculation together with \eqref{008} and  \eqref{541} shows that
\begin{equation}
\begin{split}
I_{1}&\leq C\max_{x\in [0,1]}|w_{x}+u-\theta|\int_{0}^{1}(w_{x}+u-\theta)^{2}dx\\
&\leq C\max_{x\in [0,1]}(w_{x}+u-\theta)^{2}+C\left(\int_{0}^{1}(w_{x}+u-\theta)^{2}dx \right)^{2}\\
&\leq \frac{1}{6}\int_{0}^{1}\frac{(w_{xx}+u_x-\theta_x)^{2}}{u}dx+C\int_{0}^{1}(w_{x}+u-\theta)^{2}dx+C\left(\int_{0}^{1}(w_{x}+u-\theta)^{2}dx \right)^{2},
\end{split}
\end{equation}
\begin{equation}
\begin{split}
I_{2}&\leq \frac{1}{12}\int_{0}^{1}\frac{(w_{xx}+u_x-\theta_x)^{2}}{u}dx+C\int_{0}^{1}(w_{x}+u-\theta)^{2}u_{x}^{2}dx\\
&\leq \frac{1}{12}\int_{0}^{1}\frac{(w_{xx}+u_x-\theta_x)^{2}}{u}dx+C\max_{x\in [0,1]}(w_{x}+u-\theta)^{2}\\
&\leq \frac{1}{6}\int_{0}^{1}\frac{(w_{xx}+u_x-\theta_x)^{2}}{u}dx+C\int_{0}^{1}(w_{x}+u-\theta)^{2}dx,\\
I_{3}&\leq \frac{1}{6}\int_{0}^{1}\frac{(w_{xx}+u_x-\theta_x)^{2}}{u}dx+C\int_{0}^{1}(\theta^{\beta}\theta_{x})^{2}dx.
\end{split}
\end{equation}
Thus, by \eqref{wxuth2}, we have
\begin{equation}
\begin{split}
&\frac{d}{dt}\int_{0}^{1}(w_{x}+u-\theta)^{2}dx+\int_{0}^{1}\left(\frac{(w_{xx}+u_x-\theta_x)^{2}}{u}+\frac{\theta (w_{x}+u-\theta)^{2}}{u} \right)dx\\
&\leq C\int_{0}^{1}(w_{x}+u-\theta)^{2}dx+C\left(\int_{0}^{1}(w_{x}+u-\theta)^{2}dx \right)^{2}+C\int_{0}^{1}(\theta^{\beta}\theta_{x})^{2}dx.\label{7781}
\end{split}
\end{equation}
Now we have to estimate $\int_{0}^{1}(\theta^{\beta}\theta_{x})^{2}dx$. Multiplying \eqref{003} by $(\theta^{(\beta+2)/2}-\alpha_{2}^{(\beta+2)/2})_{+}\theta^{\beta/2}$, and integrating over $[0,1]$, we get
\begin{equation}\label{th2bthx}
\begin{split}
&\frac{1}{\beta+2}\frac{d}{dt}\int_{0}^{1}(\theta^{(\beta+2)/2}-\alpha_{2}^{(\beta+2)/2})_{+}^{2}dx
+(\beta+1)\int_{0}^{1}\frac{\theta^{2\beta}\theta_{x}^{2}}{u}1_{(\theta>\alpha_{2})}dx\\
&=\int_{0}^{1}\frac{(w_{x}+u-\theta)^{2}}{u}(\theta^{(\beta+2)/2}-\alpha_{2}^{(\beta+2)/2})_{+}\theta^{\beta/2}dx+
\frac{\beta}{2}\alpha_{2}^{\frac{\beta+2}{2}}\int_{0}^{1}\frac{\theta^{\frac{3\beta}{2}}\theta_{x}^{2}}{u\theta}1_{(\theta>\alpha_{2})}dx\\
&\quad +\int_{0}^{1}\frac{(\theta^{(\beta+2)/2}-\alpha_{2}^{(\beta+2)/2})_{+}\theta^{\frac{\beta+2}{2}}(w_{x}+u-\theta)}{u}dx\\
&=J_{1}+J_{2}+J_{3}.
\end{split}
\end{equation}
On the other hand, by \eqref{008} and Young's inequality, we check that
\begin{equation*}
\begin{split}
J_{1}&\leq \int_{0}^{1}(w_{x}+u-\theta)^{2}(\theta^{\beta+1}-\alpha_{2}^{\beta+1})_{+}dx\\
&\leq \max_{x\in [0,1]}(\theta^{\beta+1}-\alpha_{2}^{\beta+1})_{+}\int_{0}^{1}(w_{x}+u-\theta)^{2}dx\\
&\leq \frac{\beta +1}{4}\int_{0}^{1}\frac{\theta^{2\beta}\theta_{x}^{2}}{u}1_{(\theta>\alpha_{2})}dx+C\left(\int_{0}^{1}(w_{x}+u-\theta)^{2}dx \right)^{2},\\
J_{2}&\leq \frac{\beta+1}{4}\int_{0}^{1}\frac{\theta^{2\beta}\theta_{x}^{2}}{u}1_{(\theta>\alpha_{2})}dx+
C\int_{0}^{1}\frac{\theta^{\beta}\theta_{x}^{2}}{u\theta^{2}}dx,\\
J_{3}&\leq C\max_{x\in [0,1]}\left(\theta^{\frac{\beta+2}{2}}-\alpha_{2}^{\frac{\beta+2}{2}} \right)_{+}
\int_{0}^{1}\theta^{\frac{\beta+2}{2}}|w_{x}+u-\theta|1_{(\theta>\alpha_{2})}dx\\
&\leq C\max_{x\in [0,1]}\left(\theta^{\frac{\beta+2}{2}}-\alpha_{2}^{\frac{\beta+2}{2}} \right)_{+}^{2}+
C\left(\int_{0}^{1}\theta^{\frac{\beta+2}{2}}|w_{x}+u-\theta|1_{(\theta>\alpha_{2})}dx\right)^{2}\\
&\leq C\int_{0}^{1}\theta^{\beta}\theta_{x}^{2}dx+C\left(\int_{0}^{1}|w_{x}+u-\theta|dx \right)^{2}\\
&\quad +C\left(\int_{0}^{1}\left(\theta^{\frac{\beta+2}{2}}-\alpha_{2}^{\frac{\beta+2}{2}}\right)_{+}|w_{x}+u-\theta|1_{(\theta>\alpha_{2})}dx\right)^{2}\\
&\leq C\int_{0}^{1}\theta^{\beta}\theta_{x}^{2}dx+C\int_{0}^{1}(w_{x}+u-\theta)^{2}dx\\
&\quad +C\int_{0}^{1}\left(\theta^{\frac{\beta+2}{2}}-\alpha_{2}^{\frac{\beta+2}{2}}\right)_{+}^{2}dx\int_{0}^{1}(w_{x}+u-\theta)^{2}dx.
\end{split}
\end{equation*}
Consequently, it follows from \eqref{th2bthx} that
\begin{equation}
\begin{split}
&\frac{2}{\beta+2}\frac{d}{dt}\int_{0}^{1}\left(\theta^{\frac{\beta+2}{2}}-\alpha_{2}^{\frac{\beta+2}{2}}\right)_{+}^{2}dx+(\beta+1)
\int_{0}^{1}\frac{\theta^{2\beta}\theta_{x}^{2}}{u}1_{(\theta>\alpha_{2})}dx\\
&\leq C\int_{0}^{1}\theta^{\beta}\theta_{x}^{2}dx+C\int_{0}^{1}\left(\theta^{\frac{\beta+2}{2}}-\alpha_{2}^{\frac{\beta+2}{2}}\right)_{+}^{2}dx
\int_{0}^{1}(w_{x}+u-\theta)^{2}dx\\
&\quad +C\int_{0}^{1}(w_{x}+u-\theta)^{2}dx+C\left(\int_{0}^{1}(w_{x}+u-\theta)^{2}dx\right)^{2},
\end{split}
\end{equation}
which, by Gronwall's inequality, together with \eqref{008}, \eqref{7487} and \eqref{7412} indicates
\begin{equation}
\begin{split}
&\int_{0}^{1}\left(\theta^{\frac{\beta+2}{2}}-\alpha_{2}^{\frac{\beta+2}{2}}\right)_{+}^{2}dx+
\int_{0}^{t}\int_{0}^{1}\theta^{2\beta}\theta_{x}^{2}1_{(\theta>\alpha_{2})}dxds\\
&\leq C+C\int_{0}^{t}\left(\int_{0}^{1}(w_{x}+u-\theta)^{2}dx \right)^{2}ds.
\end{split}
\end{equation}
Thus
\begin{gather}
\int_{0}^{1}\theta^{\beta+2}dx+
\int_{0}^{t}\int_{0}^{1}\theta^{2\beta}\theta_{x}^{2}dxds\leq C+C\int_{0}^{t}\left(\int_{0}^{1}(w_{x}+u-\theta)^{2}dx \right)^{2}ds.\label{1424}
\end{gather}
Integrating \eqref{7781} over $[0,t]$, by \eqref{7412} and \eqref{1424}, we find that
\begin{equation}
\begin{split}
&\int_{0}^{1}(w_{x}+u-\theta)^{2}dx+\int_{0}^{t}\int_{0}^{1}\Bigl( (w_{x}+u-\theta)_{x}^{2}+\theta(w_{x}+u-\theta)^{2} \Bigl) dxds\\
&\leq C+C\int_{0}^{t}\biggl( \int_{0}^{1}(w_{x}+u-\theta)^{2}dx \biggl)^{2}ds.\\
\end{split}
\end{equation}
Then, by Gronwall's inequality, we establish \eqref{009}. Furthermore, by \eqref{1424} and \eqref{7412},
\begin{gather}
\int_{0}^{1}\theta^{\beta+2}dx+
\int_{0}^{T}\int_{0}^{1}\theta^{2\beta}\theta_{x}^{2}dxds\leq C.\label{14241}
\end{gather}
\end{proof}

\begin{lemma}
There exist  positive constants $C$ and $N$ depending only on initial data and $\beta$ such that for any $T>0$,
\begin{gather}
N^{-1}\leq\theta\leq N, \label{007}
\end{gather}
and
\begin{gather}
\int_{0}^{1}\theta_{x}^{2}dx+\int_{0}^{T}\int_{0}^{1}\left(\theta_{t}^{2}+\theta_{x}^{2}+\theta_{xx}^{2}\right)dxdt\leq C. \label{010}
\end{gather}
\end{lemma}
\begin{proof}
Multiplying \eqref{003} by $\theta^{\beta}\theta_{t}$ and integrating over $[0,1]$, by Young's inequality, we get
\begin{equation}
\begin{split}
&\frac{1}{2}\frac{d}{dt}\int_{0}^{1}\frac{\theta^{2\beta}\theta_{x}^{2}}{u}dx+
\int_{0}^{1}\left(\frac{\theta^{2\beta+1}\theta_{x}^{2}}{2u^{2}}+ \theta^{\beta}\theta_{t}^{2}\right) dx\\
&=-\frac{1}{2}\int_{0}^{1}\frac{\theta^{2\beta}\theta_{x}^{2}}{u^{2}}(w_{x}+u-\theta)dx+
\frac{1}{2}\int_{0}^{1}\frac{\theta^{2\beta}\theta_{x}^{2}}{u}dx\\
&\quad +\int_{0}^{1}\frac{(w_{x}+u-\theta)^{2}\theta^{\beta}\theta_{t}}{u}dx+\int_{0}^{1}\frac{(w_{x}+u-\theta)\theta^{\beta+1}\theta_{t}}{u}dx\\
&\leq C\max_{x\in [0,1]}(w_{x}+u-\theta)^{2}+C\left(\int_{0}^{1}\theta^{2\beta}\theta_{x}^{2}dx  \right)^{2}+C\int_{0}^{1}\theta^{2\beta}\theta_{x}^{2}dx\\
&\quad +\frac{1}{2}\int_{0}^{1}\theta^{\beta}\theta_{t}^{2}dx+C\max_{x\in [0,1]}(w_{x}+u-\theta)^{4},\label{41}
\end{split}
\end{equation}
where we have used \eqref{008}, \eqref{009} and \eqref{1424}.
On the other hand, one can see that
\begin{equation}
\begin{split}
\int_{0}^{T}\max_{x\in [0,1]}(w_{x}+u-\theta)^{2}dt\leq C\int_{0}^{T}\int_{0}^{1} \biggl ((w_{x}+u-\theta)^{2}+(w_{x}+u-\theta)_{x}^{2}\biggl)dxdt\leq C,\label{54}
\end{split}	
\end{equation}
and
\begin{equation}
\begin{split}
&\int_{0}^{T}\max_{x\in [0,1]}(w_{x}+u-\theta)^{4}dt\\
&\leq C\int_{0}^{T}\int_{0}^{1}(w_{x}+u-\theta)^{4}dxdt+C\int_{0}^{T}\int_{0}^{1}|w_{x}+u-\theta|^{3}|(w_{x}+u-\theta)_{x}|dxdt\\
&\leq C\int_{0}^{T}\max_{x\in [0,1]}(w_{x}+u-\theta)^{2}dt\\
&\quad +C\int_{0}^{T}\max_{x\in [0,1]}(w_{x}+u-\theta)^{2}\int_{0}^{1}|w_{x}+u-\theta||(w_{x}+u-\theta)_{x}|dxdt\\
&\leq C+\frac{1}{2}\int_{0}^{T}\max_{x\in [0,1]}(w_{x}+u-\theta)^{4}dt+C\int_{0}^{T}\int_{0}^{1}(w_{x}+u-\theta)^{2}dx\int_{0}^{1}(w_{x}+u-\theta)_{x}^{2}dxdt\\
&\leq C+\frac{1}{2}\int_{0}^{T}\max_{x\in [0,1]}(w_{x}+u-\theta)^{4}dt,
\end{split}
\end{equation}
which implies that
\begin{gather}
\int_{0}^{T}\max_{x\in [0,1]}(w_{x}+u-\theta)^{4}dt\leq C.\label{56}
\end{gather}
By Gronwall's inequality, along with \eqref{008},\eqref{563}, \eqref{54} and \eqref{56}, we deduce from \eqref{14241} that
\begin{equation}\label{the2bthx2}
\int_{0}^{1}\theta^{2\beta}\theta_{x}^{2}dx+\int_{0}^{T}\int_{0}^{1}\left(\theta^{2\beta+1}\theta_{x}^{2}+ \theta^{\beta}\theta_{t}^{2}\right) dxdt\leq C.
\end{equation}
Noticing that
\begin{equation}
\begin{split}
&\max_{x\in[0,1]}\left|\theta^{\beta+1}-\bar{\theta}^{\beta+1}\right|\leq (\beta+1)\int_0^1 \theta^\beta |\theta_x|dx\\
&\leq C\left(\int_0^1 (\theta^\beta \theta_x)^2dx\right)^{1/2}\\
&\leq C,\label{0101}
\end{split}
\end{equation}
as a result, combining this with \eqref{563}, we establish \eqref{007}.

Finally, one can rewrite \eqref{003} as
\begin{gather}
\theta_{t}=\frac{(w_{x}+u-\theta)^{2}}{u}+\frac{(w_{x}+u-\theta)\theta}{u}+\frac{\theta^{\beta}\theta_{xx}}{u}+\frac{\beta\theta^{\beta-1}\theta_{x}^{2}}{u}
-\frac{\theta^{\beta}\theta_{x}u_{x}}{u^{2}},
\end{gather}
and then it follows from  \eqref{7412}, \eqref{007}, \eqref{56} and \eqref{the2bthx2} that
\begin{equation}
\begin{split}
\int_{0}^{T}\int_{0}^{1}\theta_{xx}^{2}dxdt&\leq C\int_{0}^{T}\int_{0}^{1}\Bigl(\theta_{t}^{2}+(w_{x}+u-\theta)^{4}+(w_{x}+u-\theta)^{2}
+\theta_{x}^{4}+\theta_{x}^{2}u_{x}^{2} \Bigr)dxdt\\
&\leq C+C\int_{0}^{T}\int_{0}^{1}(\theta_{x}^{4}+\theta_{x}^{2}u_{x}^{2})dxdt\\
&\leq C+C\int_{0}^{T}\max_{x\in [0,1]}\theta_{x}^{2}dt\\
&\leq C+\frac{1}{2}\int_{0}^{T}\int_{0}^{1}\theta_{xx}^{2}dxdt,
\end{split}
\end{equation}
which together with \eqref{the2bthx2} and \eqref{007} gives \eqref{010}.
\end{proof}

\section{Stability of the solutions}
In this section, we will focus on the stability of the solution.
\begin{lemma}\label{0001}
We have
\begin{gather}
\lim_{t\to \infty}\|w\|_{L^{2}[0,1]}=0.\label{11}
\end{gather}

\end{lemma}
\begin{proof}
By \eqref{008}, \eqref{009} and \eqref{007},
\begin{gather}
\int_{0}^{1}w_{x}^{2}dx\leq C.
\end{gather}
Multiplying \eqref{002} by $w$ and integrating over $[0,1]$, we have
\begin{equation}\label{w2a1}
\begin{split}
&\frac{1}{2}\frac{d}{dt}\int_{0}^{1}w^{2}dx+\int_{0}^{1}w^{2}dx=-\int_{0}^{1}\frac{(w_{x}+u-\theta)w_{x}}{u}dx\\
&\leq C\left(\int_{0}^{1}(w_{x}+u-\theta)^{2} dx\right)^{\frac{1}{2}}\left(\int_{0}^{1}w_{x}^{2}dx \right)^{\frac{1}{2}}\\
&\leq C\left(\int_{0}^{1}(w_{x}+u-\theta)^{2} dx\right)^{\frac{1}{2}}.
\end{split}
\end{equation}
Hence, by \eqref{009} and Young's inequality, we conclude that

\begin{gather}
\begin{split}
&\int_{0}^{1}w^{2}dx \leq Ce^{-t}\int_{0}^{t}e^{s}\left(\int_{0}^{1}(w_{x}+u-\theta)^{2} dx\right)^{\frac{1}{2}}ds \\
&\leq Ce^{-t}\int_{0}^{\frac{t}{2}}e^{s}\left(\int_{0}^{1}(w_{x}+u-\theta)^{2} dx\right)^{\frac{1}{2}}ds+
Ce^{-t}\int_{\frac{t}{2}}^{t}e^{s}\left(\int_{0}^{1}(w_{x}+u-\theta)^{2} dx\right)^{\frac{1}{2}}ds \\
&\leq Ce^{-\frac{t}{2}}+C\left(\int_{\frac{t}{2}}^{t}\int_{0}^{1}(w_{x}+u-\theta)^{2}dxds \right)^{\frac{1}{2}}, \\
\end{split}
\end{gather}
which, together with \eqref{009}, gives \eqref{11}.
\end{proof}

\begin{lemma}\label{0002}
It holds that
\begin{gather}
\lim_{t\to \infty}u(x,t)=\lim_{t \to \infty}\int_{0}^{1}\theta dx=A,\label{45}
\end{gather}
uniformly in $[0,1]$, where $A$ is given by \eqref{A0}.
\end{lemma}
\begin{proof}
Integrating \eqref{002} over $[0,x]$, and using \eqref{1212}, we obtain
\begin{gather}
\frac{\partial}{\partial t}\int_{0}^{x}w(\xi,t)d\xi+\int_{0}^{x}w(\xi,t)d\xi=\frac{w_{x}+u-\theta}{u},\label{17}
\end{gather}
which, by \eqref{001}, can be rewritten as
\begin{gather}
u_t+\left(1-\frac{\partial}{\partial t}\int_{0}^{x}w(\xi,t)d\xi\right)u=\theta+u\int_{0}^{x}w(\xi,t)d\xi.\label{171}
\end{gather}
Hence,
\begin{equation}
\begin{split}
u(x&,t)=e^{-t}u_{0}(x) \exp\biggl \{\int_{0}^{x}\bigl(w(\xi,t)-w_{0}(\xi)\bigr)d\xi \biggr\}\\
&\quad +e^{-t}\int_{0}^{t}e^{s}\exp\biggl\{\int_{0}^{x}\bigl(w(\xi,t)-w(\xi,s)\bigr)d\xi \biggr\}\biggl(\theta(x,s)+u(x,s)\int_{0}^{x}w(\xi,s)d\xi \biggr)ds\\
&=e^{-t}u_{0}(x) \exp\biggl \{\int_{0}^{x}\bigl(w(\xi,t)-w_{0}(\xi)\bigr)d\xi \biggr\}\\
&\quad +e^{-t}\int_{0}^{t}e^{s}\left(\exp\biggl\{\int_{0}^{x}\bigl(w(\xi,t)-w(\xi,s)\bigr)d\xi \biggr\}-1\right)\theta(x,s)ds\\
&\quad +e^{-t}\int_{0}^{t}e^{s}\exp\biggl\{\int_{0}^{x}\bigl(w(\xi,t)-w(\xi,s)\bigr)d\xi \biggr\}\biggl(u(x,s)\int_{0}^{x}w(\xi,s)d\xi \biggr)ds\\
&\quad +e^{-t}\int_{0}^{t}e^{s}\left(\theta(x,s)-\int_{0}^{1}\theta(x,s)dx \right)ds+e^{-t}\int_{0}^{t}e^{s}\int_{0}^{1}\theta(x,s)dxds\\
&\triangleq\sum_{i=1}^{4} I_i+e^{-t}\int_{0}^{t}e^{s}\int_{0}^{1}\theta(x,s)dxds.\label{6}
\end{split}
\end{equation}
Now, we will claim that $I_i=o(1)$ for $i=1,2,3,4$, where $o (1)$ denotes the function which converges to zero uniformly in $[0,1]$ as $t\to \infty$.
It is clear that the relation $I_1=o(1)$ is a direct result of \eqref{w2}, and $I_4=o(1)$ holds because of  \eqref{010}.

Noticing that by \eqref{17}, \eqref{008} and \eqref{1212},
\begin{gather}
\frac{\partial}{\partial t}\int_{0}^{x}w(\xi,t)d\xi+\int_{0}^{x}w(\xi,t)d\xi\leq |\int_0^xw_{\xi\xi}+u_\xi-\theta_\xi d\xi|\leq \|w_{xx}+u_x-\theta_x\|_{L^2{[0,1]}},
\end{gather}
and then, similar to what we have done in the proof of Lemma \ref{0001} (see \eqref{w2a1}), we deduce from \eqref{009} that
\begin{gather}
\int_{0}^{x}w(\xi,t)d\xi=o (1),\label{1}
\end{gather}
together with \eqref{w2} and \eqref{007}, yields
\begin{equation}\label{wth1}
\begin{split}
I_2&=e^{-t}\int_{0}^{\frac{t}{2}}e^{s}\left(\exp\biggl\{\int_{0}^{x}\bigl(w(\xi,t)-w(\xi,s)\bigr)d\xi \biggr\}-1\right)\theta(x,s)ds\\
&\quad +e^{-t}\int_{\frac{t}{2}}^{t}e^{s}\left(\exp\biggl\{\int_{0}^{x}\bigl(w(\xi,t)-w(\xi,s)\bigr)d\xi \biggr\}-1\right)\theta(x,s)ds\\
&=o(1), \,\,\, \text{as}\,\,\, t\rightarrow \infty.
\end{split}
\end{equation}
Similarly, it follows from \eqref{w2}, \eqref{008} and \eqref{1} that $I_3=o(1)$.

As a result, we have
\begin{equation}
u(x,t)=e^{-t}\int_{0}^{t}e^{s}\int_{0}^{1}\theta(x,s)dxds+o(1).\label{7021}
\end{equation}

The proof will be completed if we show the following facts:
\begin{gather}
\overline{u}\triangleq \int_{0}^{1}u(x,t)dx=u(x,t)+o (1),\label{3}
\end{gather}
\begin{gather}
\int_{0}^{1}\theta(x,t)dx=u(x,t)+o(1),\label{4}
\end{gather}
\begin{gather}
e^{-t}\int_{0}^{t}e^{s}\left(\int_{0}^{1}u(x,s)dx \right)^{2}ds-u^2(x,t)=o(1).\label{5}
\end{gather}
Indeed, by \eqref{004}, we get
\begin{equation}\label{vwe1}
\begin{split}
v(x,t)=&w(x,t)+\int_{0}^{1}v_{0}(\xi)d\xi +\int_{0}^{x}u(\xi,t)d\xi -\int_{0}^{1} \int_{0}^{\xi}u(\eta,t)d\eta d\xi\\
=&w(x,t)+\int_{0}^{1}v_{0}(\xi)d\xi+(x-\frac{1}{2})\int_{0}^{1}u(\xi,t)d\xi+o(1),
\end{split}
\end{equation}
where we have utilize the following fact
\begin{equation*}
x\int_{0}^{1}u(\xi,t)d\xi-\int_{0}^{x}u(\xi,t)d\xi=o(1),
\end{equation*}
due to \eqref{3}.

Now by virtue of \eqref{7021},  \eqref{E0}, \eqref{vwe1}, \eqref{11} and \eqref{5}, we arrive at
\begin{equation}
\begin{split}
u(x,t)=&E_{0}-\frac{1}{2}e^{-t}\int_{0}^{t}e^{s}\int_{0}^{1}v^{2}(x,s)dxds+o(1)\\
=&E_{0}-\frac{1}{2}\left(\int_{0}^{1}v_{0}(x)dx \right)^{2}-\frac{e^{-t}}{24}\int_{0}^{t}e^{s}\left(\int_{0}^{1}u(x,s)dx \right)^{2}ds+o(1)\\
=&E_{0}-\frac{1}{2}\left(\int_{0}^{1}v_{0}(x)dx \right)^{2}-\frac{1}{24}u^{2}(x,t)+o(1),\label{155}
\end{split}
\end{equation}
which, together with \eqref{4} gives \eqref{45}.

It remains to prove \eqref{3}-\eqref{5}. It is clear that \eqref{3} is a direct result of \eqref{7021}.

Multiplying \eqref{002} by $w$, adding it to \eqref{003} and integrating over $[0,1]$, together with \eqref{001}, we have
\begin{equation}
\begin{split}
&\frac{d}{dt} \int_{0}^{1}\left(\frac{1}{2}w^{2}+\theta \right)dx+\int_{0}^{1}(w^{2}+\theta)dx\\
&=\frac{d}{dt}\int_{0}^{1}udx+\int_{0}^{1}udx.\label{89}
\end{split}
\end{equation}
Next, multiplying \eqref{17} by $u$ and integrating over $[0,1]$ show that
\begin{equation}
\begin{split}
&\frac{d}{dt}\int_{0}^{1}u\int_{0}^{x}wd\xi dx+\int_{0}^{1}u\int_{0}^{x}wd\xi dx+\int_{0}^{1}(w^{2}+\theta)dx\\
&=\frac{d}{dt}\int_{0}^{1}udx+\int_{0}^{1}udx.\label{99}
\end{split}
\end{equation}
Combining \eqref{89} and \eqref{99} lead to
\begin{gather}
\frac{d}{dt}\int_{0}^{1}\left(\frac{1}{2}w^{2}+\theta \right)dx=\frac{d}{dt}\int_{0}^{1}u\int_{0}^{x}wd\xi dx+\int_{0}^{1}u\int_{0}^{x}wd\xi dx.
\end{gather}
Consequently,
\begin{equation}
\begin{split}
\int_{0}^{1}\left(\frac{1}{2}w^{2}+\theta \right)dx&=\int_{0}^{1}u\int_{0}^{x}wd\xi dx+e^{-t}\int_{0}^{t}e^{s}\int_{0}^{1}\left(\frac{1}{2}w^{2}+\theta \right)dxdt+o(1),\\
\end{split}
\end{equation}
which, together with \eqref{008}, \eqref{11}, \eqref{1} and \eqref{7021} gives \eqref{4}.

Finally, multiplying \eqref{001} by $e^t\int_{0}^{1}u(x,t)dx$ , and integrating over $[0,1]\times [0,t]$, we get
\begin{equation}
\begin{split}
e^{-t}&\int_{0}^{t}e^{s}\left(\int_{0}^{1}u(x,s)dx \right)^{2}ds-\left(\int_{0}^{1}u(x,t)dx \right)^{2} \\
&=e^{-t}\left(\int_{0}^{1}u_{0}(x)dx \right)^{2}-2e^{-t}\int_{0}^{t}e^{s}\int_{0}^{1}u(x,s)dx\int_{0}^{1}w_{x}(x,s)dxds\\
&\triangleq J_1+J_2.
\end{split}
\end{equation}
Clearly, $J_1=o(1)$ as $t\to \infty$. For $J_2$, by \eqref{008},
\begin{equation}
\begin{split}
&|J_2|=\left|e^{-t}\int_{0}^{t}e^{s}\int_{0}^{1}u(x,s)dx\int_{0}^{1}w_{x}(x,s)dxds \right|\\
&\leq Ce^{-t}\int_{0}^{t}e^{s}\int_{0}^{1}\left(\left|w_{x}+u-\theta \right|+\left|u-\int_{0}^{1}udx \right|+\left|\theta-\int_{0}^{1}\theta dx \right| \right)dxds\\
&\quad +Ce^{-t}\int_{0}^{t}e^{s}\left|\int_{0}^{1}\left(\theta-u \right)dx \right|ds\\
&\leq Ce^{-t}\int_{0}^{t}e^{s}\left(\int_{0}^{1}(w_{x}+u-\theta)^{2}+\theta_x^2 dx\right)^{\frac{1}{2}}ds+Ce^{-t}\int_{0}^{t}e^{s}\left|\int_{0}^{1}\left(\theta-u \right)dx \right|ds,
\end{split}
\end{equation}
 which, together with \eqref{009}, \eqref{010}, \eqref{3} and \eqref{4}, yields $J_2=o(1)$. Therefore, along with \eqref{3}, we obtain \eqref{5} and complete the proof.
\end{proof}

\begin{lemma}\label{0003}
We have
\begin{gather}
\lim_{t \to \infty} \theta = A,\label{55}
\end{gather}
uniformly in $[0,1]$.
\end{lemma}
\begin{proof}
It follows from \eqref{010}  and \eqref{007} that
\begin{equation}
\begin{split}
&\int_{0}^{T}\int_{0}^{1}\left(\theta^{\beta+1}-\bar{\theta}^{\beta+1} \right)^{2}dxdt\\
&\leq C\int_{0}^{T}\left(\int_{0}^{1}\theta^{\beta}|\theta_{x}|dx \right)^{2}dt\\
&\leq C\int_{0}^{T}\int_{0}^{1}\theta_{x}^{2}dxdt \leq C,
\end{split}
\end{equation}
and
\begin{equation}
\begin{split}
&\int_{0}^{T}\left|\frac{d}{dt}\int_{0}^{1}(\theta^{\beta+1}-\bar{\theta}^{\beta+1})^{2}dx \right|dt\\
&\leq C\int_{0}^{T}\int_{0}^{1}(\theta^{\beta+1}-\bar{\theta}^{\beta+1})^{2}dxdt+C\int_{0}^{T}\int_{0}^{1}(\theta^{\beta}\theta_{t}^{2}+\bar{\theta}_{t}^{2})dxdt\\
&\leq C.
\end{split}
\end{equation}
Consequently,
\begin{gather}
\lim_{t\to \infty}\int_{0}^{1}(\theta^{\beta+1}-\bar{\theta}^{\beta+1})^{2}dx=0.\label{69}
\end{gather}
On the other hand, by \eqref{007} and \eqref{010} again , we have
\begin{equation}
\begin{split}
&\max_{x\in [0,1]}(\theta^{\beta+1}-\bar{\theta}^{\beta+1})^{2}\\
&\leq C\int_{0}^{1}\left|\theta^{\beta+1}-\bar{\theta}^{\beta+1} \right|\left|(\theta^{\beta+1})_{x} \right|dx\\
&\leq C\left(\int_{0}^{1}(\theta^{\beta+1}-\bar{\theta}^{\beta+1})^{2}dx \right)^{\frac{1}{2}}\left(\int_{0}^{1}(\theta^{\beta}\theta_{x})^{2}dx \right)^{\frac{1}{2}}\\&\leq C\left(\int_{0}^{1}(\theta^{\beta+1}-\bar{\theta}^{\beta+1})^{2}dx \right)^{\frac{1}{2}},
\end{split}
\end{equation}
which together with \eqref{69} and \eqref{45} gives \eqref{55}.
\end{proof}

\begin{lemma}\label{0004}
It holds that
\begin{gather}\label{uxwxthx20}
\int_{0}^{1}\bigl(u_{x}^{2}+w_{x}^{2}+\theta_{x}^{2}\bigr)dx=o(1),
\end{gather}
and \begin{gather}
\lim_{t \to \infty} w = 0 ,\label{wt0}
\end{gather}
uniformly in $[0,1]$.
\begin{proof}
We deduce from \eqref{24} and \eqref{007} that
\begin{gather}
\int_{0}^{1}\left(w-\frac{u_{x}}{u} \right)^{2}dx\leq Ce^{-t}\int_{0}^{t}e^{s}\int_{0}^{1}\left(\theta_{x}^{2}+ w^{2} \right)dxds,
\end{gather}
which, along with \eqref{11} \eqref{008} and  \eqref{010}, implies that
\begin{gather}\label{ux2o}
\int_{0}^{1}u_{x}^{2}=o(1).
\end{gather}
Next, introducing
\begin{gather*}z(x,t)=w(x,t)+\int_{0}^{x}\Bigl(u(\xi,t)-\theta(\xi,t)\Bigl)d\xi,\end{gather*}
and by \eqref{001}-\eqref{11212}, one can check that $z(x,t)$ satisfies
\begin{gather}
z_{t}+z=\left(\frac{z_{x}}{u} \right)_{x}-\int_{0}^{x}\frac{z_{\xi}w_{\xi}}{u}d\xi-\frac{\theta^{\beta}\theta_{x}}{u},\label{7454}
\end{gather}
and the boundary condition
\begin{gather}
z_{x}(0,t)=z_{x}(1,t)=0.
\end{gather}
Multiplying \eqref{7454} by $z_{xx}$ and integrating over $[0,1]$, together with \eqref{008}, \eqref{007} and \eqref{w2}, we get
\begin{equation}
\begin{split}
&\frac{1}{2}\frac{d}{dt}\int_{0}^{1}z_{x}^{2}dx+\int_{0}^{1}z_{x}^{2}dx+\int_{0}^{1}\frac{z_{xx}^{2}}{u}dx\\
&=\int_{0}^{1}\frac{z_{x}u_{x}z_{xx}}{u^{2}}dx+\int_{0}^{1}z_{xx}\int_{0}^{x}\frac{z_{\xi}w_{\xi}}{u}d\xi dx+
\int_{0}^{1}\frac{\theta^{\beta}\theta_{x}z_{xx}}{u}dx\\
& =\int_{0}^{1}\frac{(uw+u_{x})z_{x}z_{xx}}{u^{2}}dx-
\int_{0}^{1}z_{xx}\int_{0}^{x}\left(\frac{wz_{\xi\xi}}{u}-\frac{wz_{\xi}u_{\xi}}{u^{2}} \right)d\xi dx+\int_{0}^{1}\frac{\theta^{\beta}\theta_{x}z_{xx}}{u}dx\\
& =\int_{0}^{1}\frac{(uw+u_{x})z_{x}z_{xx}}{u^{2}}dx+
\int_{0}^{1}z_{x}\left(\frac{wz_{xx}}{u}-\frac{wz_{x}u_{x}}{u^{2}} \right)dx+\int_{0}^{1}\frac{\theta^{\beta}\theta_{x}z_{xx}}{u}dx\\
&\leq\frac{1}{4}\int_{0}^{1} \frac{z_{xx}^{2}}{u}dx+
C\int_{0}^{1}\bigl(w^{2}+u_{x}^{2} \bigr)z_{x}^{2}dx+C\int_{0}^{1}\theta_{x}^{2}dx\\
&\leq
\frac{1}{4}\int_{0}^{1}\frac{z_{xx}^{2}}{u}dx+C\int_{0}^{1}(w^{2}+u_{x}^{2} )dx \max_{x\in[0,1]}z_{x}^{2}+C\int_{0}^{1}\theta_{x}^{2}dx\\
&\leq \frac{1}{4}\int_{0}^{1}\frac{z_{xx}^{2}}{u}dx+C\int_{0}^{1}(w^{2}+u_{x}^{2} )dx \left(\int_{0}^{1}z_{x}^{2}dx+\int_{0}^{1}\frac{z_{xx}^{2}}{u}dx\right)
+C\int_{0}^{1}\theta_{x}^{2}dx.\label{651}
\end{split}
\end{equation}
As a result, by \eqref{008}, \eqref{11} and \eqref{ux2o}, there exists a $T_0> 0$ such that for any $t\geq T_0$,
\begin{equation}
\frac{d}{dt}\int_{0}^{1}z_{x}^{2}dx+\int_{0}^{1}z_{x}^{2}dx+\int_{0}^{1}\frac{z_{xx}^{2}}{u}dx
\leq C\int_{0}^{1}\theta_{x}^{2}dx,\label{65}
\end{equation}
which, together \eqref{009} and \eqref{010}, indicates that
\begin{gather}\label{zx2o}
\int_{0}^{1}z_{x}^{2}dx=o(1).
\end{gather}
Observing that $z_x=w_x+u-\theta$, by \eqref{45}, \eqref{55} and \eqref{zx2o}, we obtain
\begin{gather}\label{wx20}
\int_{0}^{1}w_{x}^{2}dx=o(1).
\end{gather}
By \eqref{41}, \eqref{008} and \eqref{007}, we derive
\begin{equation}
\begin{split}
&\frac{d}{dt}\int_{0}^{1}\frac{\theta^{2\beta}\theta_{x}^{2}}{u}dx+\int_{0}^{1}\frac{\theta^{2\beta}\theta_{x}^{2}}{u}dx\\
&\leq C\max_{x\in [0,1]}(w_{x}+u-\theta)^{2}+C\max_{x\in [0,1]}(w_{x}+u-\theta)^{4}+C\int_{0}^{1}\theta_{x}^{2}dx,
\end{split}
\end{equation}
which, together with \eqref{010}, \eqref{54} and \eqref{56} yields
\begin{gather}
\int_{0}^{1}\theta_{x}^{2}dx=o(1).
\end{gather}
Therefore, we establish \eqref{uxwxthx20}. Finally, due to \eqref{w0},
\begin{gather}\label{maxw0}
\max_{x\in[0,1]}|w(x,t)|=\max_{x\in[0,1]}\Bigl|w(x,t)-\int_{0}^{1}w(\xi,t)d\xi\Bigl|\leq \int_{0}^{1}|w_x|dx\leq\Biggl(\int_{0}^{1}w_{x}^{2}dx\Biggl)^{1/2},
\end{gather}
which along with \eqref{wx20} gives \eqref{wt0}. So we finish the proof of Lemma \ref{0004}.
\end{proof}
\end{lemma}

\section{Rate of convergence}
In this section, we will further study the decay rate of convergence. More precisely, we have the following conclusion.
\begin{theorem}\label{thm41}
The convergence of $(u,w,\theta)$ to $(A,0,A)$ in $W^{1,2}(0,1)$ decays exponentially, i.e., there exist positive constants $C$, $\lambda$ which depend on
$\beta$ and the initial data such that
\begin{gather}
\int_{0}^{1}\Bigl((u-A)^{2}+w^{2}+(\theta-A)^{2}+u_{x}^{2}+w_{x}^{2}+\theta_{x}^{2} \Bigr)dx\leq Ce^{-\lambda t}.
\end{gather}
\end{theorem}
In order to prove Theorem \ref{thm41}, we need  the following two Lemmas.

\begin{lemma}
There exist positive constants $C_{i}(1,2,3,4,5)$ and $T_1$ depending only on
$\beta$ and the initial data such that for any $t>T_1$,
\begin{equation}
\begin{split}
&\frac{d}{dt}\int_{0}^{1}\left(\frac{1}{2}\frac{w^{2}}{u}+\frac{1}{2}\Bigl(\frac{\theta}{u}-\log \frac{\theta}{u}-1\Bigr)+C_{1}\Bigl(w-\frac{u_{x}}{u}\Bigr)^{2}+C_{2}\theta_{x}^{2}
+C_{3}z_{x}^{2} \right)dx\\
&+C_{4}\int_{0}^{1}\left(\frac{w^{2}}{u}+\Bigl(\frac{\theta}{u}-\log \frac{\theta}{u}-1\Bigr)+\Bigl(w-\frac{u_{x}}{u}\Bigr)^{2}+w_{x}^{2}
+z_{x}^{2}+\theta_{x}^{2}+\frac{z_{xx}^{2}}{u}+\frac{\theta^\beta\theta_{xx}^2}{u}\right)dx\\
&\leq C_{5}\int_{0}^{1}\theta_{x}^{2}dx. \label{145}
\end{split}
\end{equation}
\end{lemma}
\begin{proof}
As a consequence of Lemmas \ref{0001}-\ref{0004}, for any $\delta>0$, there exists a $T(\delta)$ such that  for any $t>T(\delta)$,
\begin{gather}\label{sum1}
\int_{0}^{1}\Bigl((u-A)^{2}+w^{2}+(\theta-A)^{2}+z^{2}+u_{x}^{2}+w_{x}^{2}+\theta_{x}^{2}+z_{x}^{2} \Bigr)(x,t)dx\leq \delta.
\end{gather}
The constant $\delta$ will be given later in the process of the proof.

Multiplying  \eqref{002} by $\frac{w}{u}$, and integrating with respect to $x$ over $[0,1]$, by \eqref{001}, we get
\begin{equation}
\begin{split}
\frac{d}{dt}&\int_{0}^{1}\left(\frac{1}{2}\frac{w^{2}}{u}+\log u \right)dx+
\int_{0}^{1}\left(\frac{w^{2}}{u}+\frac{w_{x}^{2}}{u^{2}} \right)dx\\
&=\int_{0}^{1}\left(\frac{\theta w_{x}}{u^{2}}+\frac{wu_{x}z_{x}}{u^{3}}-\frac{1}{2}\frac{w^{2}w_{x}}{u^{2}} \right)dx.\label{52}
\end{split}
\end{equation}
Similarly, multiplying \eqref{003} by $\frac{1}{u}$, along with \eqref{001} leads to
\begin{gather}\label{thulogu}
\frac{1}{2}\frac{d}{dt}\int_{0}^{1}\left(\frac{\theta}{u}-\log u\right)dx-\frac{1}{2}\int_{0}^{1}\frac{w_{x}^{2}}{u^{2}}dx=-\int_{0}^{1}\frac{\theta w_{x}}{u^{2}}dx+\frac{1}{2}\int_{0}^{1}
\left(\frac{z_{x}}{u}+\frac{\theta^{\beta}\theta_{x}u_{x}}{u^{3}} \right)dx.
\end{gather}
By virtue of \eqref{23}, \eqref{v23} and \eqref{wvxu},
\begin{gather}
-\frac{1}{2}\frac{d}{dt}\int_{0}^{1}\log\theta dx+\frac{1}{2}\int_{0}^{1}\left(\frac{z_{x}^{2}}{u\theta}+
\frac{\theta^\beta\theta_{x}^{2}}{u\theta^{2}} \right)dx=-\frac{1}{2}\int_{0}^{1}\frac{z_{x}}{u}dx,
\end{gather}
which together with  \eqref{52} and \eqref{thulogu}, we get
\begin{equation}
\begin{split}
&\frac{d}{dt}\int_{0}^{1}\left(\frac{1}{2}\frac{w^{2}}{u}+\frac{1}{2}\Bigl(\frac{\theta}{u}-\log\frac{\theta}{u}-1 \Bigr) \right)dx+\int_{0}^{1}\left(\frac{w^{2}}{u}+\frac{1}{2}\frac{w_{x}^{2}}{u^{2}}+\frac{1}{2}\frac{z_{x}^{2}}{u\theta}+
\frac{1}{2}\frac{\theta^\beta\theta_{x}^{2}}{u\theta^{2}} \right)dx\\
&=\int_{0}^{1}\left(\frac{wu_{x}z_{x}}{u^{3}}-\frac{1}{2}\frac{w^{2}w_{x}}{u^{2}}+\frac{1}{2}\frac{\theta^{\beta}\theta_{x}u_{x}}{u^{3}} \right)dx\\
&\leq C\delta^{\frac{1}{2}} \int_{0}^{1}\left(\Bigl(w-\frac{u_{x}}{u}\Bigr)^{2}+\frac{w^{2}}{u}+w_{x}^{2}+z_{x}^{2} \right)dx+\varepsilon \int_{0}^{1}\left(\Bigl(w-\frac{u_{x}}{u}\Bigr)^{2}+\frac{w^{2}}{u} \right)dx\\
&\quad +C(\varepsilon)\int_{0}^{1}\theta_{x}^{2}dx,\label{741}
\end{split}
\end{equation}
where $\varepsilon$ will be determined later and we have utilized \eqref{008}, \eqref{007} and the following fact
\begin{gather}
\max_{x\in[0,1]}|w(x,t)|\leq \delta^\frac{1}{2},
\end{gather}
due to \eqref{maxw0} and \eqref{sum1}.\\
For the convenience of further analysis, set $\varepsilon_1\triangleq \frac{1}{2}\min\{1,\frac{1}{M^2},\frac{1}{MN},\frac{1}{MN^{2+\beta}}\}$, by \eqref{007}, \eqref{008}, we rewrite \eqref{741} as
\begin{equation}
\begin{split}
&\frac{d}{dt}\int_{0}^{1}\left(\frac{1}{2}\frac{w^{2}}{u}+\frac{1}{2}\Bigl(\frac{\theta}{u}-\log\frac{\theta}{u}-1 \Bigr) \right)dx+\varepsilon_1\int_{0}^{1}\left(\frac{w^{2}}{u}+w_{x}^{2}+z_{x}^{2}+\theta_{x}^{2}
 \right)dx\\
&\leq A_1\delta^{\frac{1}{2}} \int_{0}^{1}\left(\Bigl(w-\frac{u_{x}}{u}\Bigr)^{2}+\frac{w^{2}}{u}+w_{x}^{2}+z_{x}^{2} \right)dx+\varepsilon \int_{0}^{1}\left(\Bigl(w-\frac{u_{x}}{u}\Bigr)^{2}+\frac{w^{2}}{u} \right)dx\\
&\quad +C(\varepsilon)\int_{0}^{1}\theta_{x}^{2}dx.\label{74}
\end{split}
\end{equation}
On the other hand, since $\theta/u$ is equipped with positive upper and lower bounds,
\begin{gather}
0\leq \frac{\theta}{u}-\log \frac{\theta}{u}-1\leq A_2(\theta-u)^{2}=A_2(z_x-w_x)^{2}\leq 2A_2(z_{x}^{2}+w_{x}^{2}),\label{75}
\end{gather}
which implies that
\begin{gather}
0\leq \int_0^1(\frac{\theta}{u}-\log \frac{\theta}{u}-1)dx\leq 2A_2\int_0^1(z_{x}^{2}+w_{x}^{2})dx. \label{751}
\end{gather}
Next, we multiply \eqref{003} by $\theta_{xx}$ and integrating over $[0,1]$, we deduce from \eqref{008}, \eqref{007}, \eqref{sum1} and Young's inequality that
\begin{equation}
\begin{split}
&\frac{d}{dt}\int_{0}^{1}\theta_{x}^{2}dx+\int_{0}^{1}\frac{\theta^{\beta}\theta_{xx}^{2}}{u}dx\\
&\leq C\int_{0}^{1}\Bigl(z_{x}^{4}+z_{x}^{2}+\theta_{x}^{4}+\theta_{x}^{2}u_{x}^{2} \Bigr)dx\\
&\leq C\delta \left( \max_{x\in [0,1]}z_{x}^{2}+\max_{x\in [0,1]}\theta_{x}^{2} \right)+C\int_{0}^{1}z_{x}^{2}dx\\
&\leq A_3\delta \int_{0}^{1}\left(\frac{z_{xx}^{2}}{u}+\frac{\theta^{\beta}\theta_{xx}^{2}}{u} \right)dx+A_4\int_{0}^{1}(z_{x}^{2}+\theta_{x}^{2})dx.\label{a20a}
\end{split}
\end{equation}
Using \eqref{24} and \eqref{007}, we have
\begin{equation}
\begin{split}
&\frac{d}{dt}\int_{0}^{1}\left(w-\frac{u_{x}}{u} \right)^{2}dx+\varepsilon_0 \int_{0}^{1}\left(w-\frac{u_{x}}{u} \right)^{2}dx\\
&\leq A_5\int_{0}^{1}\left(\theta_{x}^{2}+\frac{w^{2}}{u} \right)dx.\label{a21a}
\end{split}
\end{equation}
Now taking
\begin{gather*}
\begin{split}
&C_1=\frac{\varepsilon_1}{4A_5},\,\,\,C_2=\frac{\varepsilon_1}{A_4},\,\,\,C_3=\varepsilon_1,\,\,\, C_4=\frac{\varepsilon_1}{4}\min\left\{1,\frac{1}{A_2},\frac{2}{A_4},\frac{\varepsilon_0}{4A_5}\right\},\\
&\delta=\min\left\{\frac{\varepsilon_1^2}{16A_1^2}, \frac{A_4}{2A_3}, \frac{1}{2A_3},\frac{\varepsilon_1^2\varepsilon_0^2}{64A_1^2A_5^2}\right\},\,\,\,\varepsilon=\frac{\varepsilon_1}{4}\min\left\{1,\frac{\varepsilon_0}{4A_5}\right\},
\end{split}
\end{gather*}
and multiplying \eqref{751}, \eqref{a20a}, \eqref{a21a} and \eqref{65} by $\frac{\varepsilon_1}{4A_2}$, $C_2$, $C_1$ and $C_3$ respectively, then adding them together with \eqref{74}, we obtain \eqref{145}.
\end{proof}
If an inequality $C_{5}<C_{4}$ holds, then we can easily show the exponential $L^{2}$-decay of $(w,\theta-u,z,u_{x},w_{x},\theta_{x},z_{x})$
by use of \eqref{145}. However, one cannot expect it. Therefore we need a more delicate analysis.

Set
\begin{gather}\label{psid1}
\psi(x,t)\triangleq w(x,t)+\int_{0}^{x}\Bigl(u(\xi,t)-A\Bigr)d\xi-\int_{0}^{1}\int_{0}^{\xi}\Bigl(u(\eta,t)-A\Bigr)d\eta d\xi.
\end{gather}
Using $z(x,t)$ and $\psi(x,t)$, one can rewrite \eqref{002} and \eqref{003} as
\begin{gather}\label{ptzx}
\psi_{t}=\left(\frac{z_{x}}{u} \right)_{x},
\end{gather}
\begin{gather}\label{thzu}
\theta_{t}=\frac{z_{x}^{2}}{u}+\frac{\theta z_{x}}{u}+\left(\frac{\theta^{\beta}\theta_{x}}{u} \right)_{x}.
\end{gather}
Moreover, by the definition of $\psi$, one can immediately check that
\begin{gather} \label{psi0}
\int_{0}^{1}\psi dx=0,
\end{gather}
and
\begin{gather}\label{pzxtha}
\psi_{x}=z_{x}+(\theta-A).
\end{gather}
\begin{lemma}
we have
\begin{gather}
\frac{d}{dt}\int_{0}^{1}\left(\frac{1}{2}\psi ^{2} +A\Bigl(\frac{\theta}{A}-\log \frac{\theta}{A}-1\Bigr) \right)dx+A\int_{0}^{1}
\left(\frac{z_{x}^{2}}{u\theta}+\frac{\theta^{\beta}\theta_{x}^{2}}{u\theta^{2}} \right)dx=0,\label{654}
\end{gather}
\begin{gather}
\int_{0}^{1}\left(\frac{1}{2}\psi ^{2}+\psi_{x}^{2}+A\Bigl(\frac{\theta}{A}-\log \frac{\theta}{A}-1 \Bigr) \right)dx\leq
A_6\int_{0}^{1}\left(\frac{w^{2}}{u}+\Bigl(w-\frac{u_{x}}{u} \Bigr)^{2}+w_{x}^{2}+\theta_{x}^{2}+z_{x}^{2} \right)dx,\label{652}
\end{gather}
\end{lemma}
where the positive constant $A_6$ depends only on $\beta$ and the initial data.
\begin{proof}

Multiplying \eqref{ptzx} by $\psi$ and \eqref{thzu} by $-A\theta^{-1}$, adding up these relations and integrating with respect to
$x$ over $[0,1]$ , we get \eqref{654}.

 It remains to prove \eqref{652}. First, by \eqref{psi0} and \eqref{pzxtha},
\begin{gather}
\int_{0}^{1}(\psi^{2}+\psi_{x}^{2})dx\leq 2\int_{0}^{1}\psi_{x}^{2}dx \leq C\int_{0}^{1}\Bigl(z_{x}^{2}+(\theta-A)^{2}\Bigr)dx.\label{124}
\end{gather}
Next, since $\theta$ has positive upper and lower bounds,
\begin{gather}\label{alogtha}
0\leq \frac{\theta}{A}-\log \frac{\theta}{A}-1\leq C(\theta-A)^{2}.
\end{gather}
Consequently, in order to prove \eqref{652}, it is necessary to give the estimate of $(\theta-A)^{2}$.

By \eqref{66} and \eqref{004}, it indicates that
\begin{equation}
\begin{split}\label{666}
&2E_{0}-\left(\int_{0}^{1}v_{0}dx \right)^{2}=2\int_{0}^{1}\theta dx\\
&+\int_{0}^{1}\left(w^{2}+\Bigl(\int_{0}^{x}ud\xi-\int_{0}^{1}\int_{0}^{\xi}ud\eta d\xi \Bigr)^{2}+
2w\Bigl(\int_{0}^{x}ud\xi-\int_{0}^{1}\int_{0}^{\xi}ud\eta d\xi \Bigr)  \right)dx.
\end{split}
\end{equation}
Now we deduce from  \eqref{A0}, \eqref{666}, \eqref{008} and \eqref{75} that
\begin{equation}
\begin{split}
(\theta-A)^{2}&=\left(\theta+12-2\sqrt{36+3\Biggl(2E_{0}-\Bigl(\int_{0}^{1}v_{0}(\xi)d\xi\Bigr)^{2}\Biggr)} \right)^{2}\\
&\leq C\left(\theta^{2}+24\theta-24E_{0}+12\Biggl(\int_{0}^{1}v_{0}(\xi)d\xi\Biggr)^{2} \right)^{2}\\
&\leq C\left(\theta^{2}-12\int_{0}^{1}\Biggl(\int_{0}^{x}ud\xi-\int_{0}^{1}\int_{0}^{\xi}ud\eta d\xi \Biggr)^{2}dx\right)^{2}+
C\int_{0}^{1}\bigl(w^{2}+\theta_{x}^{2}\bigr)dx\\
&\leq C(\theta-u)^{2}+C\int_{0}^{1}\bigl(w^{2}+u_{x}^{2}+\theta_{x}^{2}\bigr)dx\\
&\leq C(z_{x}^{2}+w_{x}^{2})+C\int_{0}^{1}\left(w^{2}+\Bigl(w-\frac{u_{x}}{u}\Bigr)^{2}+\theta_{x}^{2} \right)dx,
\label{121}
\end{split}
\end{equation}
where in the second inequality we have taken advantage of the following fact
 \begin{equation}
\begin{split}
&\left(\theta^{2}-12\int_{0}^{1}\Biggl(\int_{0}^{x}ud\xi-\int_{0}^{1}\int_{0}^{\xi}ud\eta d\xi \Biggr)^{2}dx\right)^{2}\\
&=\left(\theta^{2}-12\int_{0}^{1}\Biggl(\int_{0}^{x}(u-\bar{u})d\xi-\int_{0}^{1}\int_{0}^{\xi}(u-\bar{u})d\eta d\xi+(x-\frac{1}{2})\bar{u} \Biggr)^{2}dx\right)^{2}\\
&\leq C\left|\theta^{2}-12\Bigl(\int_{0}^{1}udx \Bigr)^{2}\int_{0}^{1}\Bigl(x-\frac{1}{2} \Bigr)^{2}dx  \right|^{2}
+C\int_{0}^{1}u_{x}^{2}dx\\
&\leq C\left(\theta-\int_{0}^{1}udx \right)^{2}+C\int_{0}^{1}u_{x}^{2}dx\\
&\leq C\left(\theta-u\right)^{2}+C\int_{0}^{1}u_{x}^{2}dx.
\end{split}
\end{equation}

Together with \eqref{124}, \eqref{alogtha} and \eqref{121}, we establish \eqref{652} and complete the proof.
\end{proof}
{\bf Proof of Theorem \ref{thm41} and Theorem \ref{thm:light}.}
Introducing
\begin{equation}
\begin{split}
\varphi(t)\triangleq\int_{0}^{1}&\Biggl(\frac{1}{2}\frac{w^{2}}{u}+\frac{1}{2}\Bigl(\frac{\theta}{u}-\log \frac{\theta}{u}-1 \Bigr)
+C_{1}\Bigl(w-\frac{u_{x}}{u} \Bigr)^{2}+C_{2}\theta_{x}^{2}+C_{3}z_{x}^{2} \\
&+C_{6}\biggl(\frac{1}{2}\psi^{2}+A\Bigl(\frac{\theta}{A}-\log \frac{\theta}{A}-1 \Bigr) \biggr) \Biggr)dx,
\end{split}
\end{equation}
where $\psi(x,t)$ is given by \eqref{psid1} and $C_6=\frac{2C_5MN^{2+\beta}}{A}$.

Because of \eqref{008} and \eqref{007}, $\varphi(t)$ is equivalent to the norm $\parallel u-A,w,\theta-A \parallel_{1,2}^{2}$.
Setting $\lambda=\frac{C_4}{2}\min\left \{2,\frac{1}{C_{1}},\frac{1}{C_{2}},\frac{1}{C_{3}},\frac{1}{A_6C_6}\right\}$,
it follows from that \eqref{145}, \eqref{654} and \eqref{652} that
\begin{gather}
\frac{d}{dt}\varphi(t)+\lambda \varphi(t)\leq 0,
\end{gather}
holds for $t>T_1$, so we finish the proof of Theorem \ref{thm41}.

 By using the original time variable and unknown functions, Theorem \ref{thm:light} is completely proved due to  Theorem \ref{thm41}.


\begin{thebibliography}{99}
	\bibitem {1}
	Amosov, A. A., Zlotnik, A. A.: Global generalized solutions of the equations of the
	one-dimensional motion of a viscous heat-conducting gas. Soviet Math. Dokl. $\mathbf{38}$ (1989), 1-5.
	
	\bibitem {2} Amosov, A. A., Zlotnik, A.A.: Solvability ¡°in the large ¡± of a system of equations of
	the one-dimensional motion of an inhomogeneous viscous heat-conducting gas. Math.
	Notes {\bf 52} (1992), 753-763.
	
	\bibitem {3} Antontsev, S. N., Kazhikhov, A. V., Monakhov, V. N.: Boundary Value Problems
	in Mechanics of Nonhomogeneous Fluids. North-Holland, Amsterdam, 1990.
	
	\bibitem {4} Batchelor, G. K.: An Introduction to Fluid Dynamics. Cambridge University Press, London, 1967.
	

	\bibitem{2202} Huang B., Shi X. D.:  Nonlinearly exponential stability of compressible Navier-Stokes system with degenerate heat-conductivity.  J. Differential Equations $\mathbf{268}$ (2020),    2464-2490.

	\bibitem{1212} Huang, B., Shi, X. D.,    Sun, Y.: Global strong solution to compressible Navier-Stokes system with degenerate heat conductivity
	and density-depending viscosity. Commun. Math. Sci. {\bf 18} (2020), 973-985.

	\bibitem {7}
	Duan, R., Guo, A., Zhu, C. J.: Global strong solution to compressible Navier-Stokes
	equations with density dependent viscosity and temperature dependent heat conductivity. J. Differential Equations $\mathbf{262}$ (2017), 4314-4335.
	

	
	\bibitem {24}
	Jenssen, H. K., Karper, T. K.: One-dimensional compressible flow with temperature dependent transport coefficients, SIAM Journal on Mathematical Analysis, $\mathbf{ 42}$ (2010),
	904-930.
	
	\bibitem {ssyy22}
	Jiang, S.: On initial-boundary value problems for a viscous, heat-conducting, one-dimensional real gas, J. Differential Equations $\mathbf{110}$ (1994), 157-181.
	\bibitem {ssyy33}
	Jiang,  S.: On the asymptotic behavior of the motion of a viscous, heat-conducting, one-dimensional real gas, Math. Z. $\mathbf{216}$ (1994), 317-336.
	
	\bibitem {Kaz} Kazhikhov, A. V.: To a theory of boundary value problems for equations of one-dimensional
	nonstationary motion of viscous heat-conduction gases, in: Boundary Value Problem for Hydrodynamical Equations,
	No.50, Institute of Hydrodynamics, Siberian Branch Acad. USSR, 1981,37-62, in Russian.
	
    \bibitem{kawa}Kawashima,  S.: Large-time behavior of solution to the free boundary problem for the equation of a viscous heat-conductive gas,
    preprint


	\bibitem {10}
	Kawashima, S., Nishida, T.: Global solutions to the initial value problem for the equations of one dimensional motion of viscous polytropic gases. J. Math. Kyoto Univ. $\mathbf{21}$ (1981),
	825-837.
	
	\bibitem {ssyy11}
	Kawohl, B.: Global existence of large solutions to initial-boundary value problems for a viscous, heat-conducting,
	one-dimensional real gas, J. Differential Equations $\mathbf{58}$ (1985), 76-103.
	
	\bibitem {9}
	Kazhikhov, A. V., Shelukhin, V. V.: Unique global solution with respect to time of
	initial boundary value problems for one-dimensional equations of a viscous gas. J. Appl.
	Math. Mech. $\mathbf{41}$ (1977), 273-282.

	\bibitem{LIJING}
    Li, J.,   Liang, Z. L.: Some uniform estimates and large time behavior of solutions to one dimensional compressible navier stokes system
    in unbounded domain with large data. Arch. Rat. Mech. Anal. 220(2016), 1195-1208.

	\bibitem{oka}Okada,  M.: Free boundary value problems for the equations of one-dimensional motion of compressible
    viscous fluids, Japan J. Appl. Math. 4, 219-235(1987)
	
	\bibitem{220}
	Okada, M.,  Kawashima, S.: On the equations of one-dimensional motion of compressible viscous fluids. J.Math.Kyoto Univ. $\mathbf{23(1983)}$  55-71.
	
	\bibitem {13}	
Nash, J.:  Le probl\`{e}me de Cauchy pour les \'{e}quations diff\'{e}rentielles d'un fluide g\'{e}n\'{e}ral. Bull. Soc. Math. France. \textbf{90}, 487-497 (1962)
	
	
	\bibitem {15} Nagasawa, T.: On the asymptotic behavior of the one-dimensional motion of the poly-
	tropic ideal gas with stress-free condition. Q. Appl. Math. $\mathbf{46}$ (1988), 665-679.
	\bibitem {18} Nagasawa, T.: On the one-dimensional motion of the polytropic ideal gas non-fixed on
	the boundary. J. Differ. Equ. $\mathbf{65}$  (1986), 49-67.
	
	
	\bibitem {17} Nishida, T.: Equations of motion of compressible viscous fluids. Pattern and Waves
	(Nishida, T., Mimura, M., Fujii, H. Eds.). Kinokuniya/North-Holland, Amsterdam,
	pp. 97-128, 1986
	
	
	\bibitem {28}
	Pan, R. H., Zhang, W. Z.: Compressible Navier-Stokes equations with temperature dependent heat conductivities, Commun. Math. Sci. $\mathbf{13}$ (2015), 401-425.
	
	\bibitem {23}  Serrin, J.: \textit{Mathematical Principles of Classical Fluid Mechanics.} Handbuch der Physik (herausgegeben von S. Fl\"{u}gge), Bd. 8/1, Str\"{o}mungsmechanik I (Mitherausgeber C. Truesdell), 125-263. Springer, Berlin-G\"{o}ttingen-Heidelberg,  1959
	
	\bibitem {tak} Takeyuki, N.: On the outer pressure problem of the one-dimension of Polytropic ideal gas, Japan J. Appl. Math.
	$\mathbf{5}$ (1988), 53-85.
	
	
	\bibitem {ssyy44}
	Wang, D. H.: Global solutions of the Navier-Stokes equations for viscous compressible flows, Nonlinear Anal. {\bf 52} (2003), 1867-1890.
\end{thebibliography}
\end{document}